\documentclass[12pt]{article}
\usepackage{a4,latexsym,amsmath,amssymb,amsthm,enumerate,eucal,
  stmaryrd,subfigure}

\usepackage{graphicx}
\newtheorem{theorem}{Theorem}
\newtheorem{lemma}[theorem]{Lemma}
\newtheorem{proposition}[theorem]{Proposition}
\newtheorem{claim}{Claim}
\newtheorem{corollary}[theorem]{Corollary}
\newtheorem{conjecture}[theorem]{Conjecture}

\newtheorem{observation}[theorem]{Observation}

\newcommand\Setx[1] {\left\{{#1}\right\}}
\newcommand\Set[2] {\left\{{#1}:\,{#2}\right\}}

\newcommand\size[1] {\left|{#1}\right|}

\newcommand{\AAA}{\mathcal{A}}
\newcommand{\BB}{\mathcal{B}}

\newcommand{\PP}{\mathcal{P}}
\newcommand{\PPP}{\mathbb{P}}

\newcommand{\RR}{\mathcal{R}}
\renewcommand{\SS}{\mathcal{S}}

\newcommand{\inner}[3]{{#1}[{#3}]^{#2}}
\newcommand{\barr}[1]{\overline {#1}}
\newcommand{\subst}[2]{{#1}|{#2}}
\newcommand{\sdpos}[1]{\AAA_+(#1)}
\newcommand{\sdneg}[1]{\AAA_-(#1)}
\newcommand{\sdposs}[2]{\AAA_+(#1;#2)}
\newcommand{\sdnegg}[2]{\AAA_-(#1;#2)}
\newcommand{\leqx}{\preceq_L}

\newcommand{\sub}{\subseteq}

\newcommand{\conn}[1]{\mathit{conn}(#1)}
\newcommand{\iconn}[1]{\mathit{conn}_2(#1)}
\newcommand{\gr}[1]{\mathit{Gr(#1)}}

\newcommand{\case}[2]{\par\medskip\textbf{Case #1:}
  \emph{#2}\par\smallskip}

\newcommand{\fig}[1]{\includegraphics[page=#1]{5-conn-figs}}
\newcommand{\sfig}[2]{\begin{subfloat}\fig{#1}\caption{#2}\end{subfloat}}
\newcommand{\sfignarrow}[3]{
    \begin{subfloat}\hspace*{#3}\fig{#1}\hspace*{#3}%
    \caption{#2}\end{subfloat}}

\newcommand{\claimproofend}{\hspace*{.1mm}\hspace{\fill}$\triangle$}
\newenvironment{claimproof}{}{\claimproofend\par\vspace{2mm}}

\newbox\subfigbox
\makeatletter
\newenvironment{subfloat}
{\def\caption##1{\gdef\subcapsave{\relax##1}}
  \let\subcapsave=\@empty
  \let\sf@oldlabel=\label
  \def\label##1{\xdef\sublabsave{\noexpand\label{##1}}}
  \let\sublabsave\relax
  \setbox\subfigbox\hbox
  \bgroup}
{\egroup
  \let\label=\sf@oldlabel
  \subfigure[\subcapsave]{\box\subfigbox\sublabsave}}
\makeatother
\newcommand{\hfilll}{\hspace*{0mm}\hspace{\fill}\hspace*{0mm}}

\title{\textbf{Hamilton cycles\\in 5-connected line graphs}} %
\author{Tom\'{a}\v{s} Kaiser$^{\:1,2}$\\
  Petr Vr\'{a}na$^{\:1,3}$}

\date{}

\begin{document}
\maketitle

\footnotetext[1]{Department of Mathematics and Institute for
  Theoretical Computer Science, University of West Bohemia,
  Univerzitn\'{\i}~8, 306~14~Plze\v{n}, Czech Republic. Supported by
  project 1M0545 and Research Plan MSM 4977751301 of the Czech
  Ministry of Education.}

\footnotetext[2]{Partially supported by grant GA\v{C}R 201/09/0197 of
  the Czech Science Foundation. E-mail: \texttt{kaisert@kma.zcu.cz}.}

\footnotetext[3]{E-mail: \texttt{vranap@kma.zcu.cz}.}

\begin{abstract}
  A conjecture of Carsten Thomassen states that every 4-connected line
  graph is hamiltonian. It is known that the conjecture is true for
  7-connected line graphs. We improve this by showing that any
  5-connected line graph of minimum degree at least 6 is
  hamiltonian. The result extends to claw-free graphs and to
  Hamilton-connectedness.
\end{abstract}


\section{Introduction}
\label{sec:intro}

Is there a positive constant $C$ such that every $C$-connected graph
is hamiltonian? Certainly not, as shown by the complete bipartite
graphs $K_{n,n+1}$, where $n$ is large. The situation may change,
however, if the problem is restricted to graphs not containing a
specified forbidden induced subgraph. For instance, for the class of
\emph{claw-free} graphs (those not containing an induced $K_{1,3}$),
Matthews and Sumner~\cite{bib:MS-hamiltonian} conjectured the
following in 1984:

\begin{conjecture}[Matthews and Sumner]\label{conj:MS}
  Every 4-connected claw-free graph is hamiltonian.
\end{conjecture}

The class of claw-free graphs includes all line graphs. Thus,
Conjecture~\ref{conj:MS} would in particular imply that every
4-connected line graph is hamiltonian. This was stated at about the
same time as a separate conjecture by
Thomassen~\cite{bib:Tho-reflections}:
\begin{conjecture}\label{conj:thomassen}
  Every 4-connected line graph is hamiltonian.
\end{conjecture}

Although formally weaker, Conjecture~\ref{conj:thomassen} was shown to
be equivalent to Conjecture~\ref{conj:MS} by
Ryj\'{a}\v{c}ek~\cite{bib:Ryj-closure}. Several other statements are
known to be equivalent to these conjectures, including the Dominating
Cycle Conjecture~\cite{bib:Fle-cycle,bib:FJ-note}; for more work
related to these equivalences, see
also~\cite{bib:Bro-contractible,bib:KS-cycles,bib:Koch-equivalence}.

Conjectures~\ref{conj:MS} and~\ref{conj:thomassen} remain open. The
best general result to date in the direction of
Conjecture~\ref{conj:thomassen} is due to
Zhan~\cite{bib:Zha-hamiltonian} and B. Jackson (unpublished):

\begin{theorem}\label{t:zhan}
  Every 7-connected line graph is hamiltonian.
\end{theorem}

In fact, the result of \cite{bib:Zha-hamiltonian} shows that any
7-connected line graph $G$ is \emph{Hamilton-connected} --- it
contains a Hamilton path from $u$ to $v$ for each choice of distinct
vertices $u,v$ of $G$.

For 6-connected line graphs, hamiltonicity has been proved only for
restricted classes of graphs
\cite{bib:HTW-hamilton,bib:Zha-hamiltonicity}. Many papers investigate
the Hamiltonian properties of other special types of line graphs; see,
e.g., \cite{bib:Lai-every,bib:Lai-every2} and the references given
therein.

The main result of the present paper is the following improvement of
Theorem~\ref{t:zhan}:
\begin{theorem}\label{t:main}
  Every 5-connected line graph with minimum degree at least 6 is
  hamiltonian.
\end{theorem}
This provides a partial result towards
Conjecture~\ref{conj:thomassen}. Furthermore, the theorem can be
strengthened in two directions: it extends to claw-free graphs by a
standard application of the results of~\cite{bib:Ryj-closure}, and it
remains valid if `hamiltonian' is replaced by `Hamilton-connected'.

One of the ingredients of our method is an idea used (in a simpler
form) in~\cite{bib:Kai-short} to give a short proof of the
characterization of graphs with $k$ disjoint spanning trees due to
Tutte~\cite{bib:Tut-problem} and Nash-Williams~\cite{bib:NW-disjoint}
(the `tree-packing theorem'). It may be helpful to
consult~\cite{bib:Kai-short} as a companion to
Section~\ref{sec:partition-sequences} of the present paper.

The paper is organised as follows. In Section~\ref{sec:prelim}, we
recall the necessary preliminary definitions concerning graphs and
hypergraphs. Section~\ref{sec:quasi} introduces several notions
related to quasigraphs, a central concept of this paper. Here, we also
state our main result on quasitrees with tight complement
(Theorem~\ref{t:spanning-tree}). Sections~\ref{sec:narrow}--\ref{sec:skeletal}
elaborate the theory needed for the proof of this theorem, which is
finally given in Section~\ref{sec:proof-main}. Sections~\ref{sec:even}
and \ref{sec:claw-free} explain why quasitrees with tight complement
are important for us, by exhibiting their relation to connected
eulerian subgraphs of a graph. This relation is used in
Section~\ref{sec:claw-free} to prove the main result of this paper,
which is Theorem~\ref{t:main} and its corollary for claw-free
graphs. In section~\ref{sec:ham-conn}, we outline a way to further
strengthen this result by showing that graphs satisfying the
assumptions of Theorem~\ref{t:main} are in fact
Hamilton-connected. Closing remarks are given in
Section~\ref{sec:conclusion}.

The end of each proof is marked by $\square$. In proofs consisting of
several claims, the end of the proof of each claim is marked by
$\triangle$.


\section{Preliminaries}
\label{sec:prelim}

All the graphs considered in this paper are finite and may contain
parallel edges but no loops. The vertex set and the edge (multi)set of
a graph $G$ is denoted by $V(G)$ and $E(G)$, respectively. For
background on graph theory and any terminology which is not explicitly
introduced, we refer the reader to~\cite{bib:Die-graph}.

A \emph{hypergraph} $H$ consists of a vertex set $V(H)$ and a
(multi)set $E(H)$ of subsets of $V(H)$ that are called the
\emph{hyperedges} of $H$. We will be dealing exclusively with
\emph{3-hypergraphs}, that is, hypergraphs each of whose hyperedges
has cardinality 2 or 3. Multiple copies of the same hyperedge are
allowed. Throughout this paper, any hypergraph is assumed to be a
3-hypergraph unless stated otherwise. Furthermore, the symbol $H$ will
always refer to a 3-hypergraph with vertex set $V$. For
$k\in\Setx{2,3}$, a \emph{$k$-hyperedge} is a hyperedge of cardinality
$k$.

To picture a 3-hypergraph, we will represent a vertex by a solid dot,
a 2-hyperedge by a line as usual for graphs, and a 3-hyperedge $e$ by
three lines joining each vertex of $e$ to a point which is not a solid
dot (see Figure~\ref{fig:hypergraph}).

\begin{figure}
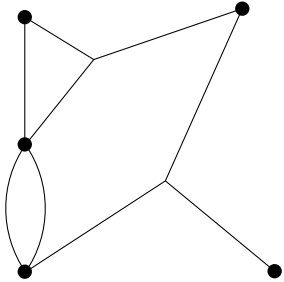

  \begin{center}
    \fig2
    \caption{A 3-hypergraph $H$ with three 2-hyperedges and two
      3-hyperedges.}
  \label{fig:hypergraph}
  \end{center}
\end{figure}

In our argument, 3-hypergraphs are naturally obtained from graphs by
replacing each vertex of degree 3 by a hyperedge consisting of its
neighbours. Conversely, we may turn a 3-hypergraph $H$ into a graph
$\gr H$: for each 3-hyperedge $e$ of $H$, we add a vertex $v_e$ and
replace $e$ by three edges joining $v_e$ to each vertex of $e$.

As in the case of graphs, the hypergraph $H$ is \emph{connected} if
for every nonempty proper subset $X\sub V$, there is a hyperedge of
$H$ intersecting both $X$ and $V-X$. If $H$ is connected, then an
\emph{edge-cut} in $H$ is any inclusionwise minimal set of hyperedges
$F$ such that $H-F$ is disconnected. For any integer $k$, the
hypergraph $H$ is \emph{$k$-edge-connected} if it is connected and
contains no edge-cuts of cardinality less than $k$. The \emph{degree}
of a vertex $v$ is the number of hyperedges incident with $v$.

To extend the notion of induced subgraph to hypergraphs, we adopt the
following definition. For $X \sub V$, we define $H[X]$ (the
\emph{induced subhypergraph of $H$ on $X$}) as the hypergraph with vertex set
$X$ and hyperedge set
\begin{equation*}
  E(H[X]) = \Set{e\cap X}{e\in E(H) \text{ and } \size{e\cap X} \geq 2}.
\end{equation*}
If $e\cap X = f\cap X$ for distinct hyperedges $e,f$, we include this
hyperedge in multiple copies. Furthermore, we assume a canonical
assignment of hyperedges of $H$ to hyperedges of $H[X]$. To stress
this fact, we always write the hyperedges of $H[X]$ as $e\cap X$,
where $e\in E(H)$.

Let $\PP$ be a partition of a set $X$. $\PP$ is \emph{trivial} if $\PP
= \Setx{X}$.  A set $Y \sub X$ is \emph{$\PP$-crossing} (or: \emph{$Y$
  crosses $\PP$}) if it intersects at least two classes of $\PP$.

As usual, another partition $\RR$ of $X$ \emph{refines} $\PP$ (written
as $\RR\leq\PP$) if every class of $\RR$ is contained in a class of
$\PP$. In this case we also say that $\RR$ is \emph{finer} than $\PP$
or that $\PP$ is \emph{coarser}. If $\RR\leq\PP$ and $\RR\neq\PP$,
then we write $\RR<\PP$ and say that $\RR$ is \emph{strictly finer}
(and $\PP$ is \emph{strictly coarser}). It is well known that the
order $\leq$ on partitions of $X$ is a lattice; the infimum of any two
partitions $\PP,\RR$ (i.e., the unique coarsest partition
that refines both $\PP$ and $\RR$) is denoted by $\PP\wedge\RR$.

If $Y\sub X$, then the partition \emph{induced on $Y$ by $\PP$}
is
\begin{equation*}
  \PP[Y] = \Set{P\cap Y}{P\in\PP \text{ and } P\cap Y\neq\emptyset}.
\end{equation*}


\section{Quasigraphs}
\label{sec:quasi}

A basic notion in this paper is that of a quasigraph. It is a
generalization of tree representations and forest representations
used, e.g., in~\cite{bib:FKK-decomposing}. 

Recall from Section~\ref{sec:prelim} that $H$ is a 3-hypergraph on
vertex set $V$. A \emph{quasigraph} in $H$ is a pair $(H,\pi)$, where
$\pi$ is a function assigning to each hyperedge $e$ of $H$ a set
$\pi(e) \sub e$ which is either empty or has cardinality 2. The value
$\pi(e)$ is called the \emph{representation} of $e$ under
$\pi$. Usually, the underlying hypergraph is clear from the context,
and we simply speak about a quasigraph $\pi$. Quasigraphs will be
denoted by lowercase Greek letters.

In this section, $\pi$ will be a quasigraph in $H$. Considering all
the nonempty sets $\pi(e)$ as graph edges, we obtain a graph $\pi^*$
on $V$. The hyperedges $e$ with $\pi(e)\neq\emptyset$ are said to be
\emph{used} by $\pi$. The set of all such hyperedges of $H$ is denoted
by $E(\pi)$. The edges of the graph $\pi^*$, in contrast, are denoted
by $E(\pi^*)$ as expected. We emphasize that, by definition, $\pi^*$
spans all the vertices in $V$.

To picture $\pi$, we use a bold line to connect the vertices of
$\pi(e)$ for each hyperedge $e$ used by $\pi$. An example of a
quasigraph is shown in Figure~\ref{fig:quasi}.

\begin{figure}
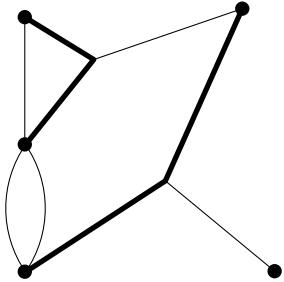

  \centering
  \fig3
  \caption{A quasigraph $\rho$ in the hypergraph of
    Figure~\ref{fig:hypergraph}.}
  \label{fig:quasi}
\end{figure}

The quasigraph $\pi$ is a \emph{acyclic} (or a \emph{quasiforest}) if
$\pi^*$ is a forest; $\pi$ is a \emph{quasitree} if $\pi^*$ is a
tree. Furthermore, we define $\pi$ to be a \emph{quasicycle} if
$\pi^*$ is the union of a cycle and a (possibly empty) set of isolated
vertices.

If $e$ is a hyperedge of $H$, then $\pi-e$ is the quasigraph obtained
from $\pi$ by changing the value at $e$ to $\emptyset$. The
\emph{complement} $\barr\pi$ of $\pi$ is the spanning subhypergraph of
$H$ comprised of all the hyperedges of $H$ not used by $\pi$. Since
$\pi$ includes the information about its underlying hypergraph $H$, it
makes sense to speak about its complement without specifying $H$
(although we sometimes do specify it for emphasis). Note that
$\barr\pi$ is not a quasigraph.

How to define an analogue of the induced subgraph for quasigraphs? Let
$X\sub V$. At first sight, a natural choice for the underlying
hypergraph of a quasigraph induced by $\pi$ on $X$ is $H[X]$. It is
clear how to define the value of the quasigraph on a hyperedge $e\cap
X$, except if $\size e = 3$ and $\size{e\cap X} = 2$ (see
Figure~\ref{fig:sect}(a)). In particular, if $\pi(e)$ intersects both
$X$ and $V-X$, then $e\cap X$ will not be used by the induced
quasigraph; furthermore, it is (at least for our purposes) not
desirable to include $e\cap X$ in the complement of the induced
quasigraph either. This brings us to the following replacement for
$H[X]$ (cf. Figure~\ref{fig:sect}(b)).

\begin{figure}
  \centering
  \sfig{26}{Possible types of 3-hyperedges $e$ with $\size{e\cap X} =
    2$ with respect to the quasigraph $\pi$.}
  \hfilll
  \sfig{27}{The corresponding 2-hyperedges of the induced
    quasigraph. Note that $e$ does not have a corresponding hyperedge.}
  \caption{An illustration to the definition of the $\pi$-section at $X$.}
  \label{fig:sect}
\end{figure}

The \emph{$\pi$-section} of $H$ at $X$ is the hypergraph $\inner H\pi
X$ defined as follows:
\begin{itemize}
\item $\inner H \pi X$ has vertex set $X$,
\item its hyperedges are the sets $e\cap X$, where $e$ is a hyperedge
  of $H$ such that $\size{e\cap X}\geq 2$ and $\pi(e) \sub X$.
\end{itemize}
The quasigraph $\pi$ in $H$ naturally determines a quasigraph $\pi[X]$
in $\inner H\pi X$, defined by
\begin{equation*}
  (\pi[X])(e\cap X) = \pi(e),
\end{equation*}
where $e\in E(H)$ and $e\cap X$ is any hyperedge of $\inner H \pi X$.
We refer to $\pi[X]$ as the quasigraph \emph{induced} by $\pi$ on
$X$. Let us stress that whenever we speak about the complement of
$\pi[X]$, it is --- in accordance with the definition --- its
complement in $\inner H \pi X$.

The ideal quasigraphs for our purposes in the later sections of this
paper would be quasitrees with connected complement. It turns out,
however, that this requirement is too strong, and that the following
weaker property will suffice. The quasigraph $\pi$ \emph{has tight
  complement} (in $H$) if one of the following holds:
\begin{enumerate}[\quad(a)]
\item $\barr\pi$ is connected, or
\item there is a partition $V = X_1\cup X_2$ such that for $i=1,2$,
  $X_i$ is nonempty and $\pi[X_i]$ has tight complement (in $\inner H
  \pi {X_i}$); furthermore, there is a hyperedge $e\in E(\pi)$ such
  that $\pi(e)\sub X_1$ and $e\cap X_2\neq\emptyset$.
\end{enumerate}
The definition is illustrated in Figure~\ref{fig:tight}.

\begin{figure}
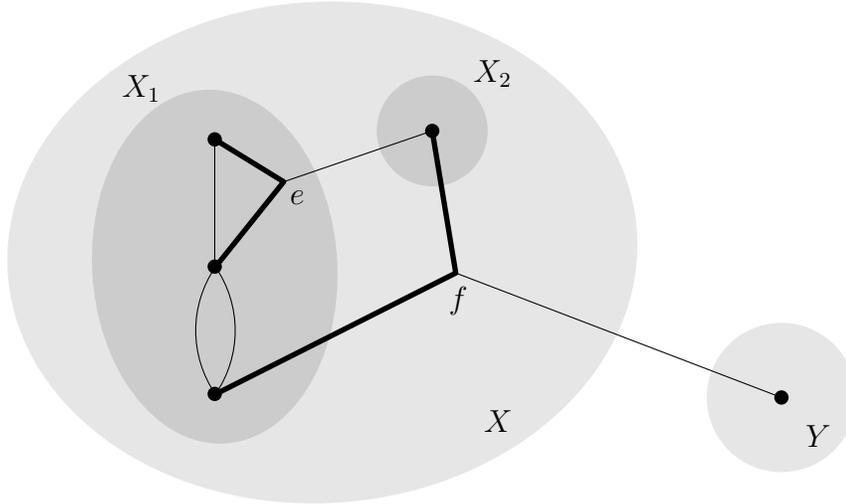

  \begin{center}
    \fig{1}
    \caption{The quasigraph $\rho$ of Figure~\ref{fig:quasi} has
      tight complement in $H$. The ovals show the subsets of $V$
      relevant to the definition of tight complement. For $i=1,2$,
      $\rho[X_i]$ has connected complement in $\inner
      H\rho {X_i}$, so $\rho[X]$ has tight complement in
      $\inner H\rho X$ `thanks to' the hyperedge $e$. Similarly,
      $f$ makes the complement of $\rho$ in $H$ tight.}
    \label{fig:tight}
  \end{center}
\end{figure}

Our main result regarding quasitrees in hypergraphs is the following:
\begin{theorem}\label{t:spanning-tree}
  Let $H$ be a 4-edge-connected 3-hypergraph. If no 3-hyperedge in $H$
  is included in any edge-cut of size 4, then $H$ contains a quasitree with
  tight complement.
\end{theorem}
Theorem~\ref{t:spanning-tree} will be proved in
Section~\ref{sec:proof-main}.

An equivalent definition of quasigraphs with tight complement is based
on the following concept. Let us say that a partition $\PP$ of $V$ is
\emph{$\pi$-narrow} if for every $\PP$-crossing hyperedge $e$ of $H$,
$\pi(e)$ is also $\PP$-crossing. (We call $\PP$ `narrow' since none of
these sets $\pi(e)$ fits into a class of $\PP$.) For instance, the
partition shown in Figure~\ref{fig:parts}(b) below is
$\pi$-narrow. Observe that the trivial partition is $\pi$-narrow for
any $\pi$.

\begin{lemma}
  \label{l:tight-narrow}
  A quasigraph $\pi$ in $H$ has tight complement if and only if there
  is no nontrivial $\pi$-narrow partition of $V$.
\end{lemma}
\begin{proof}
  We prove the `only if' part by induction on the number of vertices
  of $H$. If $\size V = 1$, the assertion is trivial. Assume that
  $\size V > 1$ and that $\PP$ is a nontrivial partition of $V$; we
  aim to prove that $\PP$ is not $\pi$-narrow. Consider the two cases
  in the definition of tight complement. If $\barr\pi$ is connected
  (Case (a)), then there is a $\PP$-crossing hyperedge $e$ of
  $\barr\pi$. Since $\pi(e) = \emptyset$ is not $\PP$-crossing, $\PP$
  is not $\pi$-narrow.

  In Case (b), there is a partition $V = X_1\cup X_2$ into nonempty
  sets such that each $\pi[X_i]$ has tight complement in $\inner H \pi
  {X_i}$. Suppose that $\PP[X_1]$ is nontrivial. By the induction
  hypothesis, it is not $\pi[X_1]$-narrow. Consequently, there is a
  hyperedge $f$ of $\inner H \pi{X_1}$ contained in $\pi[X_1]$ and such
  that $\pi(f)\sub P\cap X_1$, where $P\in\PP$. It follows that
  $\PP$ is not $\pi$-narrow as claimed. 

  By symmetry, we may assume that both $\PP[X_1]$ and $\PP[X_2]$ are
  trivial. Since $\PP$ is nontrivial, it must be that $\PP =
  \Setx{X_1,X_2}$. Case (b) of the definition of tight complement
  ensures that there is a hyperedge $e\in E(\pi)$ such that
  $\pi(e)\sub X_1$ and $e\cap X_2\neq\emptyset$. Since $e$ is
  $\PP$-crossing and $\pi(e)$ is not, $\PP$ is not $\pi$-narrow. This
  finishes the proof of the `only if' part.

  The `if' direction will be proved by contradiction. Suppose that $V$
  admits no nontrivial $\pi$-narrow partition, but $\pi$ does not have
  tight complement in $H$. Let $\RR$ be a coarsest possible partition
  of $V$ such that each $\pi[X]$, where $X\in\RR$, has tight
  complement in $\inner H\pi X$. (To see that at least one partition
  with this property exists, consider the partition of $V$ into
  singletons.)  Since $\RR$ is nontrivial by assumption, there is an
  $\RR$-crossing hyperedge $e$ of $H$ with $\pi(e) \sub R_1$, where
  $R_1$ is some class of $\RR$. Since $e$ is $\RR$-crossing, it
  intersects another class $R_2$ of $\RR$. By the definition,
  $\pi[R_1\cup R_2]$ has tight complement in $\inner H\pi{R_1\cup
    R_2}$, which contradicts the maximality of $\RR$.
\end{proof}


\section{Narrow and wide partitions}
\label{sec:narrow}

We begin this section by modifying the definition of a $\pi$-narrow
partition of $V$. If $\pi$ is a quasigraph in $H$, then a partition
$\PP$ of $V$ is \emph{$\pi$-wide} if for every hyperedge $e$ of $H$,
$\pi(e)$ is a subset of a class of $\PP$. (In particular, $\pi(e)$ is
not $\PP$-crossing for any $\PP$-crossing hyperedge $e$.) An example
of a $\pi$-wide partition is shown in Figure~\ref{fig:parts}(a) below.
Again, the trivial partition is $\pi$-wide for any $\pi$.
Lemma~\ref{l:tight-narrow} has the following easier analogue:

\begin{lemma}\label{l:tree-wide}
  If $\pi$ is a quasigraph in $H$, then $\pi^*$ is connected if and
  only if there is no nontrivial $\pi$-wide partition of $V$.
\end{lemma}
\begin{proof}
  We begin with the `only if' direction. Suppose that $\PP$ is a
  nontrivial partition of $V$. Since $\pi^*$ is a connected graph with
  vertex set $V$, there is an edge $\pi(e)$ of $\pi^*$ crossing
  $\PP$. This shows that $\PP$ is not $\pi$-wide.

  Conversely, suppose that $\pi^*$ is disconnected, and let $\PP$ be
  the partition of $V$ whose classes are the vertex sets of components
  of $\pi^*$. Let $e$ be a hyperedge of $H$. We claim that $\pi(e)$ is
  not $\PP$-crossing. This is certainly true if $e\notin E(\pi)$. In
  the other case, $\pi(e)$ is an edge of $\pi^*$ and both of its
  endvertices must be contained in the same component of $\pi^*$,
  which proves the claim. We conclude that $\PP$ is a (nontrivial)
  $\pi$-wide partition of $V$.
\end{proof}

It is interesting that both the class of $\pi$-narrow partitions and
the class of $\pi$-wide partitions are closed with respect to meets in
the lattice of partitions:
\begin{observation}\label{obs:meet}
  If $\pi$ is a quasigraph in $H$ and $\PP$ and $\RR$ are $\pi$-narrow
  partitions, then $\PP\wedge\RR$ is $\pi$-narrow. Similarly, if $\PP$
  and $\RR$ are $\pi$-wide, then $\PP\wedge\RR$ is $\pi$-wide.
\end{observation}

By Observation~\ref{obs:meet}, for any quasigraph $\pi$ in $H$, there
is a unique finest $\pi$-narrow partition of $V$, which will be
denoted by $\sdnegg\pi H$. Similarly, there is a unique finest
$\pi$-wide partition of $V$, denoted by $\sdposs\pi H$. If the
hypergraph is clear from the context, we write just $\sdpos\pi$ or
$\sdneg\pi$. Lemmas~\ref{l:tight-narrow} and \ref{l:tree-wide} provide
us with a useful interpretation of $\sdpos\pi$ and $\sdneg\pi$. It is
not hard to show from the latter lemma that the classes of $\sdpos\pi$
are exactly the vertex sets of components of $\pi^*$. Similarly, by
Lemma~\ref{l:tight-narrow}, the classes of $\sdneg\pi$ are all
maximal subsets $X$ of $V$ such that $\pi[X]$ has tight complement in
$\inner H\pi X$.

We call the classes of $\sdpos\pi$ the \emph{positive} $\pi$-parts of
$H$ and the classes of $\sdneg\pi$ the \emph{negative} $\pi$-parts of
$H$. (See Figure~\ref{fig:parts} for an illustration.) The terms
`positive' and `negative' are chosen with regard to the terminology of
photography, with `positive' used for $\pi$ and `negative' for its
complement, in accordance with the above discussion.

\begin{figure}
  \centering
  \sfignarrow8{A quasigraph $\tau$ in $H$ and the positive $\tau$-parts
    of $H$ (the gray regions).}{5mm}
  \hfilll
  \sfignarrow9{The negative $\tau$-parts of
    $H$. Note that the vertex $v$ belongs to a larger negative
    $\tau$-part, although it forms a component of $\barr{\tau}$ on its
    own.}{5mm}
  \caption{Positive and negative parts.}
  \label{fig:parts}
\end{figure}

We note the following simple corollary of Lemma~\ref{l:tight-narrow}:
\begin{lemma}\label{l:sub}
  Let $\pi$ be a quasigraph in $H$. For $i=1,2$, let $X_i\sub V$ be
  such that $\pi[X_i]$ has tight complement in $\inner H \pi
  {X_i}$. Then the following holds:
  \begin{enumerate}[\quad(i)]
  \item each $X_i$ is contained in a class of $\sdneg\pi$ (as a
    subset), and
  \item if $H$ contains a hyperedge $e$ such that $e$ intersects each
    $X_i$ and $\pi(e)\sub X_1$ (we allow $e\notin E(\pi)$), then
    $X_1\cup X_2$ is contained in a class of $\sdneg\pi$.
  \end{enumerate}
\end{lemma}
\begin{proof}
  (i) Clearly, if $\PP$ is a $\pi$-narrow partition of $V$, then
  $\PP[X_1]$ is $\pi[X_1]$-narrow; it follows that $\sdneg\pi[X_1]
  \geq \sdneg{\pi[X_1]}$. By Lemma~\ref{l:tight-narrow},
  $\sdneg{\pi[X_1]}$ is trivial. Hence $\sdneg\pi[X_1]$ is also
  trivial. A symmetric argument works for $X_2$.

  (ii) It suffices to prove that $\pi[X_1\cup X_2]$ has tight
  complement in $\inner H \pi {X_1\cup X_2}$. If not, let $\PP$ be a
  nontrivial $\pi[X_1\cup X_2]$-narrow partition of $X_1\cup X_2$. By
  the assumption, each $\PP[X_i]$ has to be trivial as it is
  $\pi[X_i]$-narrow. Thus, $\PP=\Setx{X_1,X_2}$. However, since
  $\pi(e)\subseteq X_1$, this is not a $\pi[X_1\cup X_2]$-narrow
  partition --- a contradiction.
\end{proof}

We use the partitions $\sdpos\pi$ and $\sdneg\pi$ to introduce an
order on quasigraphs. If $\pi$ and $\sigma$ are quasigraphs in $H$,
then we write
\begin{equation*}
  \pi \unlhd \sigma \text{\quad if\quad} \sdpos\pi \leq \sdpos\sigma 
  \text{ and } \sdneg\pi\leq\sdneg\sigma.
\end{equation*}
Clearly, $\unlhd$ is a partial order.

For a set $X\sub V$, let us say that two quasigraphs $\pi$ and
$\sigma$ in $H$ are \emph{$X$-similar} if the following holds for
every hyperedge $e$ of $H$:
\begin{enumerate}[\quad(1)]
\item $\pi(e)\sub X$ if and only if $\sigma(e)\sub X$, and
\item if $\pi(e)\not\sub X$, then $\pi(e) = \sigma(e)$.
\end{enumerate}
Let us collect several easy observation about $X$-similar quasigraphs:

\begin{observation}\label{obs:similar}
  If $X\sub V$ and quasigraphs $\pi$ and $\sigma$ are $X$-similar,
  then the following holds:
  \begin{enumerate}[\quad(i)]
  \item $\inner H \pi X = \inner H \sigma X$,
  \item if $X\in\sdpos\pi$, then $\sdpos\sigma \leq \sdpos\pi$,
  \item if $X\in\sdneg\pi$, then $\sdneg\sigma \leq \sdneg\pi$.
  \end{enumerate}
\end{observation}

The following lemma is an important tool which facilitates the use of
induction in our argument.
\begin{lemma}\label{l:improve}
  Let $X\sub V$ and let $\pi$ and $\sigma$ be $X$-similar quasigraphs
  in $H$. Then the following holds:
  \begin{equation*}
    \text{if\quad} \pi[X] \unlhd \sigma[X], \text{\quad then\quad} \pi
    \unlhd \sigma.
  \end{equation*}
\end{lemma}
\begin{proof}
  Note that by Observation~\ref{obs:similar}(i), $\inner H\pi X = \inner
  H\sigma X$. We need to prove that
  \begin{align}
    \text{if\quad} \sdneg{\pi[X]} \leq
    \sdneg{\sigma[X]}, &\text{\quad then\quad}
    \sdneg\pi \leq \sdneg\sigma,\label{eq:sdneg}
  \end{align}
  and an analogous assertion (\ref{eq:sdneg}$^+$) with all occurences
  of `$-$' replaced by `$+$'.

  We prove~\eqref{eq:sdneg}. By the definition of $\sdneg\sigma$,
  \eqref{eq:sdneg} is equivalent to the statement that
  \begin{align*}
    &\text{if every $\sigma[X]$-narrow partition of $X$ is
      $\pi[X]$-narrow (in $\inner H\pi X$),}\\
    &\text{then every $\sigma$-narrow partition of $V$ is
      $\pi$-narrow (in $H$).}
  \end{align*}
  Assume thus that every $\sigma[X]$-narrow partition is
  $\pi[X]$-narrow and that $\PP$ is a $\sigma$-narrow partition of
  $V$. For contradiction, suppose that $\PP$ is not $\pi$-narrow.

  We claim that $\PP[X]$ is $\sigma[X]$-narrow in $\inner H\sigma
  X$. Let $e\cap X$ be a $\PP[X]$-crossing hyperedge of $\inner H\sigma
  X$ (where $e\in E(H)$). Then $e$ is $\PP$-crossing, and since $\PP$
  is $\sigma$-narrow, $\sigma(e)$ is $\PP$-crossing. By the definition
  of $\inner H\sigma X$, $\sigma(e)\sub X$ and thus $\sigma(e) =
  \sigma[X](e\cap X)$ is $\PP[X]$-crossing. This proves the claim.

  Since every $\sigma[X]$-narrow partition of $X$ is assumed to be
  $\pi[X]$-narrow, $\PP[X]$ is $\pi[X]$-narrow. 

  On the other hand, $\PP$ is not $\pi$-narrow, so there is a
  $\PP$-crossing hyperedge $f$ of $H$ such that $\pi(f)$ is not
  $\PP$-crossing. However, $\sigma(f)$ is $\PP$-crossing as $\PP$ is
  $\sigma$-narrow. Thus, $\pi(f)\neq\sigma(f)$, and since $\pi$ and
  $\sigma$ are $X$-similar, both $\pi(f)$ and $\sigma(f)$ are subsets
  of $X$. It follows that $\sigma(f)$, and therefore also the
  hyperedge $f\cap X$ of $\inner H\sigma X = \inner H\pi X$, is
  $\PP[X]$-crossing. We have seen that $\PP[X]$ is $\pi[X]$-narrow,
  and this observation implies that $\pi(f)$ is $\PP[X]$-crossing and
  therefore $\PP$-crossing. This contradicts the choice of $f$.

  The proof of~(\ref{eq:sdneg}$^+$) is similar to the above but
  simpler. The details are omitted.
\end{proof}


\section{Partition sequences}
\label{sec:partition-sequences}

Besides the order $\unlhd$ introduced in Section~\ref{sec:narrow}, we
will need another derived order $\preceq$ on quasigraphs, one that is
used in the basic optimization strategy in our proof. Let $\pi$ be a
quasigraph in $H$. Similarly as in~\cite{bib:Kai-short}, we associate
with $\pi$ a sequence of partitions of $V$, where each partition is a
refinement of the preceding one. Since $H$ is finite, the partitions
`converge' to a limit partition whose classes have a certain
favourable property.

Recall from Section~\ref{sec:narrow} that there is a uniquely defined
partition of $V$ into positive $\pi$-parts; we will let this partition
be denoted by $\PP_0^\pi$. The \emph{partition sequence} of $\pi$ is
the sequence
\begin{equation*}
  \PPP^\pi = (\PP_0^\pi,\PP_1^\pi,\dots),
\end{equation*}
where for even (odd) $i\geq 1$, $\PP_i^\pi$ is obtained as the union
of partitions of $X$ into positive (negative, respectively)
$\pi[X]$-parts of $\inner H \pi X$ as $X$ ranges over classes of
$\PP_{i-1}^\pi$. (See Figure~\ref{fig:sequence}.) Thus, for instance,
for even $i\geq 2$ we can formally write
\begin{equation*}
  \PP_i^\pi = \bigcup_{X\in\PP_{i-1}^\pi}\sdpos{\pi[X]}.
\end{equation*}
Since $H$ is finite, we have $\PP^\pi_k = \PP^\pi_{k+2}$ for large
enough $k$, and we set $\PP^\pi_\infty = \PP^\pi_k$. 

Let us call a set $X\sub V$ \emph{$\pi$-solid} (in $H$) if $\pi[X]$ is
a quasitree with tight complement in $\inner H \pi X$.  By the
construction, any class of $\PP^\pi_\infty$ is $\pi$-solid.

\begin{figure}
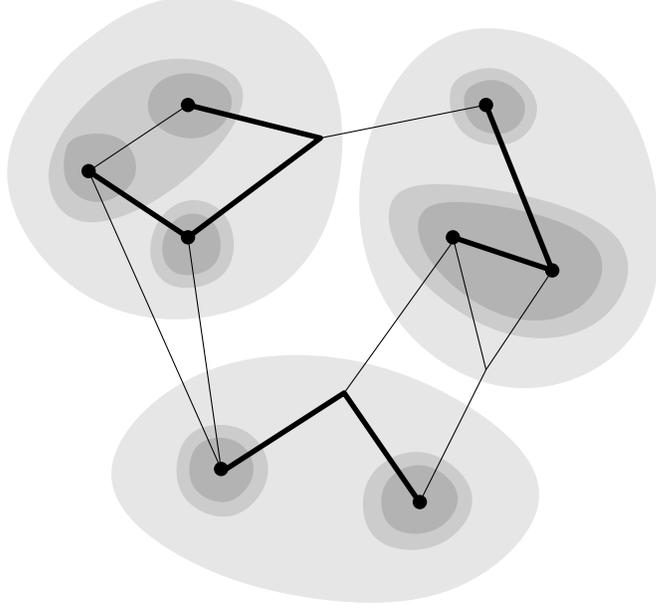

  \centering
  \fig{10}
  \caption{The partition sequence of the quasigraph $\tau$ from Figure
    \ref{fig:parts}. Partitions $\PP_0^\tau, \PP_1^\tau$ and $\PP_2^\tau$
    are shown in different gray shades from light to dark. Note that
    the classes of $\PP_2^\tau$ are $\tau$-solid.}
\label{fig:sequence}
\end{figure}

Let us define a lexicographic order on sequences of partitions: if
$(\AAA_0,\AAA_1,\dots)$ and $(\BB_0,\BB_1,\dots)$ are sequences of
partitions of $V$, write
\begin{equation*}
  (\AAA_0,\AAA_1,\dots) \leqx (\BB_0,\BB_1,\dots)
\end{equation*}
if there exists some $i$ such that for $j < i$, $\AAA_j = \BB_j$, while
$\AAA_i$ strictly refines $\BB_i$. 

We can now define the order $\preceq$ on quasigraphs as promised. Let
$\pi$ and $\sigma$ be quasigraphs in $H$. Define
\begin{equation*}
  \pi \preceq \sigma \text{\quad if\quad} \pi\unlhd\sigma\text{ and }\PPP^\pi \leqx \PPP^\sigma.
\end{equation*}
If $\pi\preceq\sigma$ but $\sigma\not\preceq\pi$, we write
$\pi\prec\sigma$.

From Lemma~\ref{l:improve}, we can deduce a similar observation
regarding the order $\preceq$ (in which the implication is actually
replaced by equivalence). 

\begin{lemma}\label{l:prec}
  Let $X\sub V$ and assume that either $X$ is a positive $\pi$-part of
  $H$, or $\PP_0^\pi$ is trivial and $X$ is a negative $\pi$-part of
  $H$. Let $\pi$ and $\sigma$ be $X$-similar quasigraphs in $H$. Then
  the following holds:
  \begin{equation*}
    \pi[X] \preceq \sigma[X] \text{\quad if and only if\quad} \pi
    \preceq \sigma.
  \end{equation*}
\end{lemma}
\begin{proof}
  We consider two cases depending on whether $X$ is a positive or
  negative $\pi$-part of $H$.  

  \case1{$X$ is a positive $\pi$-part of $H$.}  

  Since $\pi$ and $\sigma$ are $X$-similar, we have
  \begin{align}\label{eq:partseq}
    \PPP^\pi &= (\PP_0^\pi, \quad
    \PP_1^{\pi[X]} \cup \PP_1^\pi[V-X], \quad
    \PP_2^{\pi[X]} \cup \PP_2^\pi[V-X], \quad
    \dots)\text{\quad and}\notag\\
    \PPP^\sigma &= (\PP_0^\sigma, \quad
    \PP_1^{\sigma[X]} \cup \PP_1^\pi[V-X], \quad
    \PP_2^{\sigma[X]} \cup \PP_2^\pi[V-X], \quad
    \dots).
  \end{align}

  Assume first that $\pi[X] \preceq
  \sigma[X]$. Equations~\eqref{eq:partseq} imply that for each $i\geq
  1$, $\PP_i^\pi \leq \PP_i^\sigma$. Furthermore, $\pi[X] \unlhd
  \sigma[X]$ and Lemma~\ref{l:improve} imply that
  $\pi\unlhd\sigma$. In particular,
  \begin{equation*}
    \PP_0^\pi = \sdpos\pi \leq \sdpos\sigma = \PP_0^\sigma
  \end{equation*}
  so $\PPP^\pi \leqx \PPP^\sigma$ and therefore also $\pi \preceq
  \sigma$.

  Conversely, assume that $\pi\preceq\sigma$. The fact that
  $\PPP^\pi\leqx\PPP^\sigma$ together with \eqref{eq:partseq} implies
  that for $i\geq 1$, $\PP_i^{\pi[X]} \leq \PP_i^{\sigma[X]}$. Recall
  that $X$ is a positive $\pi$-part of $H$. We claim that $X$ is also
  a positive $\sigma$-part of $H$; indeed, this follows from the fact
  that $\PP_0^\pi \leq \PP_0^\sigma$ and that $\pi$ and $\sigma$ are
  $X$-similar. This claim implies
  \begin{equation}\label{eq:zero}
    \PP_0^{\pi[X]} = X = \PP_0^{\sigma[X]}
  \end{equation}
  and, consequently, $\PPP^{\pi[X]} \leqx \PPP^{\sigma[X]}$. It
  remains to verify that $\pi[X] \unlhd \sigma[X]$. This follows from
  \eqref{eq:zero} and the observation that $\PP_1^{\pi[X]} \leq
  \PP_1^{\sigma[X]}$. (Here we use the fact that if $\PP_0^\pi$ is
  trivial, then $\PP_1^\pi = \sdneg\pi$.)
  
  \case2{$\PP_0^\pi$ is trivial and $X$ is a negative $\pi$-part of
    $H$.}

  In this case, equations~\eqref{eq:partseq} are replaced by
  \begin{align}\label{eq:partseq2}
    \PPP^\pi &= (\Setx V, \quad
    \sdneg{\pi[X]} \cup \PP_1^\pi[V-X], \notag\\
    &\qquad\PP_0^{\pi[X]} \cup \PP_2^\pi[V-X], \quad
    \PP_1^{\pi[X]} \cup \PP_3^\pi[V-X], \quad
    \dots)\text{\quad and}\notag\\
    \PPP^\sigma &= (\Setx V, \quad
    \sdneg{\sigma[X]} \cup \PP_1^\pi[V-X], \notag\\
    &\qquad\PP_0^{\sigma[X]} \cup \PP_2^\pi[V-X], \quad
    \PP_1^{\sigma[X]} \cup \PP_3^\pi[V-X], \quad
    \dots).
  \end{align}
  Assume first that $\pi\preceq\sigma$. Since $X$ is a positive
  $\pi$-part of $H$, the partition $\sdneg{\pi[X]}$
  appearing in the second term of $\PPP^\pi$ is trivial. A similar
  observation holds for $\sigma$ in place of $\pi$. Hence, $\PPP^\pi$
  and $\PPP^\sigma$ are equal in their first two terms and
  \eqref{eq:partseq2} directly implies that $\PPP^{\pi[X]} \leqx
  \PPP^{\sigma[X]}$. Moreover, $\pi[X]\unlhd\sigma[X]$ is implied by
  \eqref{eq:partseq2} as well. We conclude that $\pi[X] \preceq
  \sigma[X]$.

  The converse implication follows from \eqref{eq:partseq2} without
  any further effort. The proof is complete.
\end{proof}

\begin{corollary}\label{cor:long}
  Let $\pi$ and $\sigma$ be $X$-similar quasigraphs in $H$, where $X
  \in \PP_i^\pi$ for some $i$. Then the following holds:
  \begin{equation*}
    \pi[X] \preceq \sigma[X] \text{\quad if and only if\quad} \pi
    \preceq \sigma.
  \end{equation*}
\end{corollary}
\begin{proof}
  Follows from Lemma~\ref{l:prec} by easy induction.
\end{proof}

We conclude this section by a lemma that suggests a relation between
$\preceq$-maximal and acyclic quasigraphs. If $\pi$ and $\sigma$ are
quasigraphs in $H$, then let us call $\sigma$ a \emph{restriction} of
$\pi$ if for every hyperedge $e$ of $H$, $\sigma(e)$ equals either
$\pi(e)$ or $\emptyset$.

\begin{lemma}
  \label{l:break-cycles}
  Let $\pi$ be a quasigraph in $H$ and $i\geq 0$. If $\pi[X]$ is
  acyclic for each $X\in\PP_i^\pi$, but $\pi$ itself is not acyclic,
  then there exists an acyclic restriction $\sigma$ of $\pi$ such that
  $\sigma \succ \pi$.
\end{lemma}
\begin{proof}
  Suppose that $\gamma$ is a quasicycle in $H$ such that
  $E(\gamma)\sub E(\pi)$. By the assumption, not all of the edges of
  $\gamma^*$ are contained in the same class of $\PP_i^\pi$; in other
  words, $\gamma^*$ contains a $\PP_i^\pi$-crossing edge. Let $k \geq
  0$ be the least integer such that $\gamma^*$ contains a
  $\PP_k^\pi$-crossing edge $\gamma(e)$ (where $e\in E(H)$).

  Since $\PP_0^\pi$ is a partition of $V$ into positive $\pi$-parts
  and $\gamma$ is a restriction of $\pi$, there are no
  $\PP_0^\pi$-crossing edges in $\gamma^*$. Thus, $k\geq
  1$. Similarly, if $j\geq 2$ is even and $X\in\PP_{j-1}^\pi$, then
  $\inner H \pi X$ contains no $\PP_j^\pi[X]$-crossing edges. It
  follows that $k$ is odd. Let $Y$ be the class of $\PP_{k-1}^\pi$
  containing all edges of $\gamma^*$ as subsets.
  
  Set $\rho = \pi-e$. Observe that $(\rho[Y])^*$ is a connected graph
  spanning $Y$, since $(\pi[Y])^*$ has this property, and the removal
  of the edge $\pi(e)$ cannot disconnect $(\pi[Y])^*$ as $\pi(e)$ is
  contained in a cycle in $\pi^*$. Thus, $\PP_0^\rho = \Setx{Y}$.

  Assume that $\pi(e) = z_1z_2$ and let $Z_i$ ($i=1,2$) be the class
  of $\PP_k^\pi$ containing $z_i$. Since each $Z_i$ is a class of
  $\sdneg {\pi[Y]}$, $\rho[Z_i]$ has tight complement
  in $\inner H \rho {Z_i}$. Now the hyperedge $e\cap Y$ containing
  $z_1$ and $z_2$ is not used by $\rho$. By
  Lemma~\ref{l:sub}(ii), $Z_1\cup Z_2$ is contained in a class
  of $\sdneg{\rho[Y]}$. Consequently,
  \begin{equation*}
    \sdneg{\rho[Y]} > \sdneg{\pi[Y]}
  \end{equation*}
  and therefore $\rho[Y]\succ\pi[Y]$. By Corollary~\ref{cor:long},
  $\rho \succ \pi$.

  If $\rho$ is not acyclic, we repeat the previous step. Since $H$ is
  finite, we will arrive at an acyclic restriction $\sigma\succ\pi$ of
  $\pi$ after finitely many steps.
\end{proof}


\section{Contraction and substitution}
\label{sec:subst}

In this section, we introduce two concepts related to partitions:
contraction and substitution.

Let $\PP$ be a partition of $V$. The \emph{contraction} of $\PP$ is
the operation whose result is the hypergraph $H/\PP$ defined as
follows. For $A\sub V$, define $A/\PP$ as the subset of $\PP$
consisting of all the classes $P\in\PP$ such that $A\cap
P\neq\emptyset$. The hypergraph $H/\PP$ has vertex set $\PP$ and it
hyperedges are all the sets of the form $e/\PP$, where $e$ ranges over
all $\PP$-crossing hyperedges. Thus, $H/\PP$ is a 3-hypergraph,
possibly with multiple hyperedges. As in the case of induced
subhypergraphs, each hyperedge $f$ of $H/\PP$ is understood to have an
assigned corresponding hyperedge $e$ of $H$ such that $f=e/\PP$.

\begin{figure}
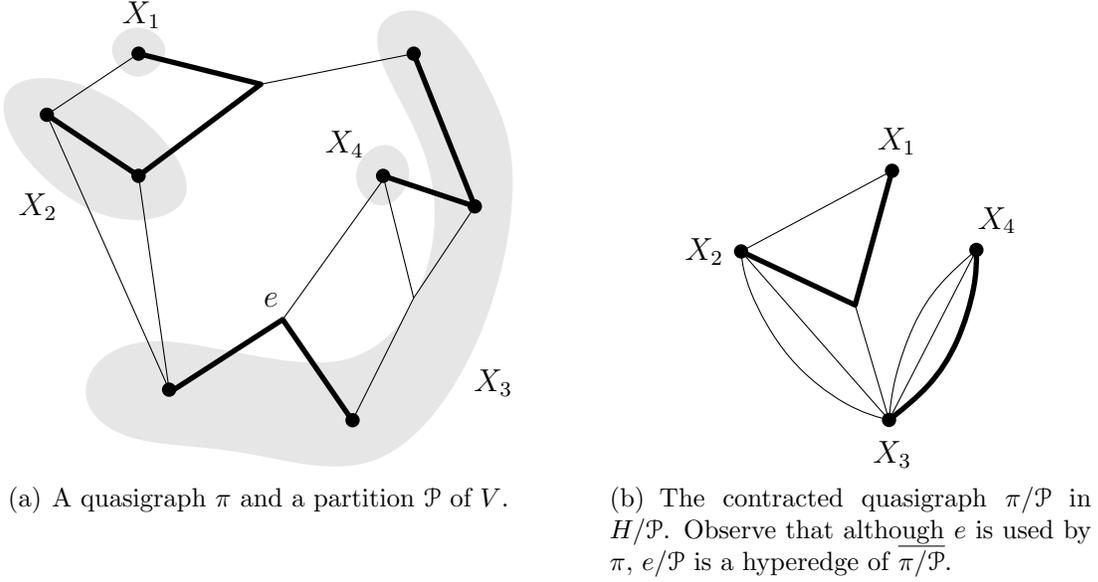

  \centering
  \sfig{13}{A quasigraph $\pi$ and a partition $\PP$ of $V$.}
  \hfilll{}
  \sfignarrow{14}{The contracted quasigraph $\pi/\PP$ in $H/\PP$. Observe that
    although $e$ is used by $\pi$, $e/\PP$ is a hyperedge of
    $\barr{\pi/\PP}$.}{1cm}
  \caption{An example of contraction.}
  \label{fig:contraction}
\end{figure}

If $\pi$ is a quasigraph in $H$, we define $\pi/\PP$ as the quasigraph
in $H/\PP$ consisting of the hyperedges $e/\PP$ such that $\pi(e)$
is $\PP$-crossing; the representation is defined by
\begin{equation*}
  (\pi/\PP)(e/\PP) = \pi(e)/\PP.
\end{equation*}
(Contraction is illustrated in Figure~\ref{fig:contraction}.)  In
keeping with our notation, the complement of $\pi/\PP$ in $H/\PP$ is
denoted by $\barr{\pi/\PP}$. Observe that if $e\in E(H)$, then $e/\PP$
is an edge of $\barr{\pi/\PP}$ if and only if $e$ is $\PP$-crossing
and $\pi(e)$ is not. The following lemma will be useful:

\begin{lemma}\label{l:cycle-di}
  Let $\RR\leq\PP$ be partitions of $V$ and $\pi$ be a quasigraph in
  $H$. If $\gamma/\RR$ is a quasicycle in $\barr{\pi/\RR}$, then one
  of the following holds:
  \begin{enumerate}[\quad(a)]
  \item for some $X\in\PP$, $\gamma[X]/\RR[X]$ is a quasicycle in
    the complement of $\pi[X]/\RR[X]$ in $\inner H \pi X/\RR[X]$,
  \item $\gamma/\PP$ is a nonempty quasigraph in $\barr{\pi/\PP}$ such
    that $(\gamma/\PP)^*$ is an eulerian graph (a graph with all
    vertex degrees even).
  \end{enumerate}
\end{lemma}
\begin{proof}
  We will use two formal equalities whose proof is left to the kind
  reader as a slightly tedious exercise: for $X\in\PP$ and any
  quasigraph $\sigma$ in $H$,
  \begin{align}
    \sigma[X]/\RR[X] &= (\sigma/\RR)[\RR[X]],\label{eq:formal-sigma}\\
    \inner H \pi X/\RR[X] &= \inner{(H/\RR)}{\pi/\RR}{\RR[X]}.\label{eq:formal-H}
  \end{align}

  Let $\gamma/\RR$ be a quasicycle in $\barr{\pi/\RR}$. Suppose that
  there is $X\in\PP$ such that every edge of $(\gamma/\RR)^*$ is a
  subset of $\RR[X]$. Let $\tilde\gamma = (\gamma/\RR)[\RR[X]]$. Thus,
  $\tilde\gamma$ is a quasicycle in $(H/\RR)[\RR[X]]$ and
  $E(\tilde\gamma)$ is disjoint from $E((\pi/\RR)[\RR[X]])$. We infer
  that $\tilde\gamma$ is a quasigraph in
  $\inner{(H/\RR)}{\pi/\RR}{\RR[X]}$. Using \eqref{eq:formal-H}, we
  find that $\tilde\gamma$ is a quasigraph in $\inner H \pi
  X/\RR[X]$. Finally, we use \eqref{eq:formal-sigma} twice (for
  $\gamma$ and $\pi$) and conclude that condition (a) holds.

  Thus, we may assume that the endvertices $Y_1,Y_2$ of some edge
  $\gamma(e)$ of $(\gamma/\RR)^*$ are classes of $\RR$ contained in
  different classes of $\PP$ (say, $X_1$ and $X_2$,
  respectively). Thus, $\gamma/\PP$ is a nonempty quasigraph in
  $H/\PP$. Furthermore, $E(\gamma/\PP)$ is clearly disjoint from
  $E(\pi/\PP)$. To verify (b), it remains to prove that
  $(\gamma/\PP)^*$ is eulerian. This is immediate from the fact that
  $(\gamma/\PP)^*$ can be obtained from the graph $(\gamma/\RR)^*$
  (which consists of a cycle and isolated vertices) by identifying
  certain sets of vertices (namely those contained in the same class
  of $\PP$).
\end{proof}

If $X\sub V$ and $\sigma$ is a quasigraph in $\inner H \pi X$, we
define the \emph{substitution} of $\sigma$ into $\pi$ as the operation
which produces the following quasigraph $\subst\pi\sigma$ in $H$:
\begin{equation*}
  (\subst\pi\sigma)(e) = 
  \begin{cases}
    \pi(e) & \text{ if $e\cap X\notin E(\inner H \pi X)$,}\\
    \sigma(e\cap X) & \text{ otherwise.}
  \end{cases}
\end{equation*}
This yields a well-defined represented subhypergraph of $H$. (See
Figure~\ref{fig:subst}.) More generally, let $\PP$ be a family of
disjoint subsets of $V$ and for each $X\in\PP$, let $\sigma_X$ be a
quasigraph in $\inner H \pi X$. Assume we substitute each $\sigma_X$
into $\pi$ in any order. For distinct $X\in\PP$, the hyperedge sets of
the hypergraphs $\inner H \pi X$ are pairwise disjoint, since $e\in
E(\inner H \pi X)$ only if $\size{e\cap X}\geq 2$. It follows easily
that the resulting hypergraph $\sigma$ in $H$ is independent of the
chosen order. This hypergraph will be denoted by
\begin{equation*}
  \sigma = \subst\pi{\Set{\sigma_X}{X\in\PP}}.
\end{equation*}

\begin{figure}
  \centering
  \sfig{15}{A quasigraph $\pi$ in $H$ and a set $X\sub V$.}
  \hfilll
  \sfig{16}{A quasigraph $\sigma$ in $\inner H \pi X$.}
  \\[8mm]
  \hfilll
  \sfig{17}{The quasigraph $\subst\pi\sigma$.}
  \hfilll
  \caption{An example of substitution.}
  \label{fig:subst}
\end{figure}

Substitution behaves well with respect to taking induced quasigraphs
and contraction:
\begin{lemma}
  \label{l:subst}
  Let $\pi$ be a quasigraph in $H$ and $\PP$ a partition of
  $V$. Suppose that for each $X\in\PP$, $\sigma_X$ is a quasigraph in
  $\inner H \pi X$, and define
  \begin{equation*}
    \sigma = \subst\pi{\Set{\sigma_X}{X\in\PP}}.
  \end{equation*}
  Then the following holds for every $Y \sub X \in \PP$:
  \begin{enumerate}[\quad(i)]
  \item $\inner H \sigma Y = \inner{(\inner H \pi X)} {\sigma_X} Y$,
  \item $\sigma[Y] = \sigma_X[Y]$.
  \end{enumerate}
  Furthermore,\noindent
  \begin{enumerate}[\quad(iii)]
  \item $\sigma/\PP = \pi/\PP$. 
  \end{enumerate}
\end{lemma}
\begin{proof}
  (i) Using the definition of $\inner H \sigma Y$ and the definition
  of substitution, it is not hard to verify that $e_0 \sub V$ is a
  hyperedge of $\inner H \sigma Y$ if and only if $e_0=e\cap Y$, where
  $e$ is a hyperedge of $H$ such that $\size{e\cap Y} \geq 2$,
  $\pi(e)\sub X$ and $\sigma_X(e\cap X)\sub Y$. If we expand the
  right hand side of the equality in (i) according to these
  definitions, we arrive at precisely the same set of conditions.

  (ii) Both sides of the equation are quasigraphs in $\inner H \sigma
  Y$. We will check that they assign the same value to a hyperedge
  $e\cap Y$ of $\inner H \sigma Y$. For such hyperedges, we have
  \begin{equation}\label{eq:subst1}
    \sigma[Y](e\cap Y) = \sigma(e) = \sigma_X(e\cap X)
  \end{equation}
  where the second equality follows from the definition of
  substitution. On the other hand, by part (i), $e\cap Y$ is a hyperedge of
  $\inner{(\inner H \pi X)}{\sigma_X}Y$, and thus
  \begin{equation}\label{eq:subst2}
    \sigma_X[Y](e\cap Y) = \sigma_X(e\cap X).
  \end{equation}
  The assertion follows by comparing~\eqref{eq:subst1} and
  \eqref{eq:subst2}.

  (iii) Both $\sigma/\PP$ and $\pi/\PP$ are quasigraphs in
  $H/\PP$. Let $e/\PP$ be a hyperedge of $H/\PP$, where $e\in
  E(H)$. Using the definitions of substitution and contraction, one
  can check that
  \begin{equation*}
    (\sigma/\PP)(e/\PP) =
    \begin{cases}
      \pi(e)/\PP & \text{if $e\cap X \notin E(\inner H \pi X)$ and
        $\pi(e)$ is $\PP$-crossing,}\\
      \sigma_X(e)/\PP & \text{if $e\cap X \in E(\inner H \pi X)$ and
        $\sigma_X(e)$ is $\PP$-crossing,}\\
      \emptyset & \text{otherwise.}
    \end{cases}
  \end{equation*}
  However, the middle case can never occur since $\sigma_X(e)\sub X$
  and $\sigma_X(e)$ is therefore not $\PP$-crossing. It follows easily
  that $(\sigma/\PP)(e/\PP) = (\pi/\PP)(e/\PP)$.
\end{proof}


\section{The Skeletal Lemma}
\label{sec:skeletal}

In this section, we prove a lemma which is a crucial piece of our
method. It leads directly to an inductive argument for the existence
of a quasitree with tight complement under suitable assumptions, which
will be given in Section~\ref{sec:proof-main}.

If $\pi$ is a quasigraph in $H$, then a partition $\PP$ of $V$ is said
to be \emph{$\pi$-skeletal} if every $X\in\PP$ is $\pi$-solid and the
complement of $\pi/\PP$ in $H/\PP$ is acyclic.

\begin{lemma}[Skeletal Lemma]\label{l:skeletal}
  Let $\pi$ be an acyclic quasigraph in $H$. Then there is an acyclic
  quasigraph $\sigma$ in $H$ such that $\sigma \succeq \pi$ and
  $\sigma$ satisfies one of the following:
  \begin{enumerate}[\quad(a)]
  \item $\sigma \succ \pi$, or
  \item there is a $\sigma$-skeletal partition $\SS$.
  \end{enumerate}
\end{lemma}
\begin{proof}
  We proceed by contradiction. Let the pair $(\pi,H)$ be a
  counterexample such that $H$ has minimal number of vertices; thus,
  no acyclic quasigraph $\sigma \succeq \pi$ in $H$ satisfies any of
  (a) and (b). Note that $\pi$ is not a quasitree with tight
  complement (which includes the case $\size{V} = 1$), for otherwise
  $\sigma = \pi$ would satisfy condition (b) with $\SS=\Setx V$.

  \begin{claim}\label{cl:components}
    $\PP_0^\pi$ is nontrivial.
  \end{claim}
  \begin{claimproof}
    Suppose the contrary and note that $\PP := \sdneg\pi$ is
    nontrivial. Consider a set $Y\in\PP$ and the acyclic quasigraph
    $\pi[Y]$. By the minimality of $H$, there is a quasigraph
    $\sigma_Y \succeq \pi[Y]$ in $\inner H \pi Y$ satisfying condition
    (a) or (b) (with respect to $\pi[Y]$ and $\inner H\pi Y$). Define
    \begin{equation*}
      \sigma = \subst\pi{\Set{\sigma_Y}{Y\in\PP}}.
    \end{equation*}
    By Lemmas~\ref{l:break-cycles} and \ref{l:subst}(ii), we may
    assume that $\sigma$ is acyclic.

    Assume first that for some $Y\in\PP$, $\sigma_Y \succ \pi[Y]$
    (case (a) of the lemma). Since $\sigma[Y] = \sigma_Y$
    (Lemma~\ref{l:subst}(ii)), Lemma~\ref{l:prec} implies that $\sigma
    \succ \pi$, a contradiction with the choice of $\pi$.

    We conclude that case (b) holds for each $Y\in\PP$, namely that
    there exists a partition $\SS_Y$ which is $\sigma_Y$-skeletal in
    $\inner H \pi Y$. Set
    \begin{equation*}
      \SS = \bigcup_{Y\in\PP} \SS_Y.
    \end{equation*}
    We claim that $\SS$ is $\sigma$-skeletal. Let $Z\in\SS$ and assume
    that $Z\sub Y\in\PP$. Since $Z$ is $\sigma_Y$-solid, and since
    $\sigma[Z] = \sigma_Y[Z]$ and $\inner H \sigma Y = \inner{(\inner
      H \pi Y)} {\sigma_Y} Z$ by Lemma~\ref{l:subst}(i)--(ii), $Z$ is
    $\sigma$-solid.

    Suppose that $\barr{\sigma/\SS}$ is not acyclic and choose a
    quasigraph $\gamma$ in $H$ such that $\gamma/\SS$ is a quasicycle
    in $\barr{\sigma/\SS}$. By Lemma~\ref{l:cycle-di}, $\gamma/\PP$ is
    a nonempty quasigraph in the complement $\barr{\pi/\PP}$ of
    $\pi/\PP$ in $H/\PP$. However, by the definition of $\sdneg\pi$,
    every $\PP$-crossing hyperedge of $H$ belongs to $\pi/\PP$ and
    thus cannot be used by $\gamma/\PP$, a contradiction. It follows
    that $\barr{\sigma/\SS}$ is indeed acyclic and $\SS$ is
    $\sigma$-skeletal. This contradiction with the choice of $\pi$
    concludes the proof of the claim.
  \end{claimproof}
  
  For each $X\in\PP_0^\pi$, $\inner H \pi X$ has fewer vertices than
  $H$. By the minimality of $H$, there is an acyclic quasigraph
  $\rho_X \succeq \pi[X]$ in $\inner H \pi X$. Define
  \begin{equation*}
    \rho = \subst \pi {\Set{\rho_X}{X\in\PP_0^\pi}}
  \end{equation*}
  By Lemma~\ref{l:prec}, $\rho \succeq \pi$. Note that since
  $\PP_0^\pi$ is $\pi$-wide, $\rho^*$ is the disjoint union of the
  graphs $\rho_X^*$ ($X\in\PP_0^\pi$). Therefore, $\rho$ is
  acyclic. 

  If $\rho_X\succ\pi[X]$ for some $X\in\PP_0^\pi$, then by
  Lemmas~\ref{l:subst}(ii) and \ref{l:prec}, $\rho \succ \pi$ and we
  have a contradiction.  Consequently, for each $X\in\PP_0^\pi$, there
  is a $\rho_X$-skeletal partition $\RR_X$ (with respect to the
  hypergraph $\inner H \pi X$). We define a partition $\RR$ of $V$ by
  \begin{align}
    \RR = \bigcup_{X\in\PP_0^\pi} \RR_X.
  \end{align}
  Similarly as in the proof of Claim~\ref{cl:components}, each
  $Y\in\RR$ is easily shown to be $\rho$-solid. An important
  difference in the present situation, however, is that $\RR$ may not
  be $\rho$-skeletal as there may be quasicycles in
  $\barr{\rho/\RR}$. Any such quasicycle $\gamma'$ can be represented
  by a quasigraph $\gamma$ in $H$ such that $\gamma' = \gamma/\RR$.

  Thus, let $\gamma$ be a quasigraph in $H$ such that $\gamma/\RR$ is
  a quasicycle in $\barr{\rho/\RR}$. By Lemma~\ref{l:cycle-di}, there
  are two possibilities:
  \begin{enumerate}[\quad(a)]
  \item for some $X\in\PP_0^\pi$, $\gamma[X]/\RR_X$ is a quasicycle in the
    complement of $\rho[X]/\RR_X$ in $\inner H \rho X/\RR_X$, or
  \item $\gamma/\PP_0^\pi$ is a nonempty quasigraph in the complement
    of $\rho/\PP_0^\pi$ in $H/\PP$ such that $(\gamma/\PP_0^\pi)^*$ is
    an eulerian graph.
  \end{enumerate}

  Since $\rho[X] = \rho_X$ (Lemma~\ref{l:subst}(ii)) and $\RR_X$ is
  $\rho_X$-skeletal, case (a) is ruled out. Thus, we can choose a
  hyperedge $f_\gamma$ of $H$ such that $\gamma(f_\gamma)$ is
  $\PP_0^\pi$-crossing. As $\gamma/\RR$ is a quasicycle in
  $\barr{\rho/\RR}$, $\rho(f_\gamma)$ is contained in a class of
  $\RR$. If $f_\gamma$ is used by $\rho$, then this class will be
  denoted by $Y_\gamma$ and we will say that the chosen hyperedge
  $f_\gamma$ is a \emph{connector for $Y_\gamma$}.

  \begin{claim}\label{cl:all-used}
    For each quasicycle $\gamma/\RR$ in $\barr{\rho/\RR}$, the
    hyperedge $f_\gamma$ is used by $\rho$.
  \end{claim}
  \begin{claimproof}
    Suppose to the contrary that $\gamma(f_\gamma) = u_1u_2$, where
    each $u_i$ ($i=1,2$) is contained in a different class $X_i$ of
    $\PP_0^\pi$. By Lemma~\ref{l:improve} and
    Observation~\ref{obs:similar}(ii), $\PP_0^\pi = \PP_0^\rho$. Let
    $\sigma$ be the quasigraph in $H$ defined by
    \begin{equation*}
      \sigma(e) =   
      \begin{cases}
        \pi(e) & \text{if $e\neq f_\gamma$,}\\
        u_1u_2 & \text{otherwise.}
      \end{cases}
    \end{equation*}
    (See Figure~\ref{fig:skeletal1}.) Considering the role of the
    hyperedge $e$, we see that
    \begin{equation}
      \label{eq:strictly}
      \PP_0^\rho < \PP_0^\sigma.
    \end{equation}

    \begin{figure}
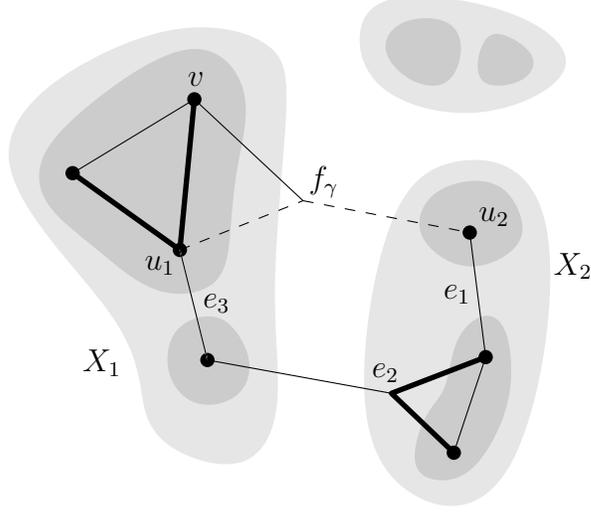

      \centering
      \hfilll
      \fig{18}
      \hfilll
      \caption{An illustration to the proof of
        Claim~\ref{cl:all-used}. Some hyperedges are omitted. The
        light gray regions are the classes of $\PP_0^\pi$, the darker
        ones are the classes of $\RR$. Bold lines indicate the
        quasigraph $\rho$. The set $\Setx{f_\gamma,e_1,e_2,e_3}$
        corresponds to a quasicycle $\gamma$ in $H/\RR$. The
        quasigraph $\sigma$ is obtained by including $f_\gamma$ in
        $E(\rho)$, with the representation given by dashed lines. Note
        that $v$ is contained in the same negative $\sigma$-part as
        $u_1$.}
      \label{fig:skeletal1}
    \end{figure}
  
    Next, we would like to prove that 
    \begin{equation}
      \label{eq:sdneg-same}
      \sdneg\rho\leq \sdneg\sigma.
    \end{equation}
    First of all, we claim that $u_1$ and $u_2$ are contained in the
    same class of $\sdneg\sigma$. Let the vertices on the unique cycle
    in $(\gamma/\RR)^*$ be $T_1,\dots,T_k$ in this order, where each
    $T_i$ is a class of $\RR$, $u_1\in T_1$ and $u_2\in T_k$. By
    symmetry, we may assume that $\size{f_\gamma\cap T_k} = 1$ (i.e.,
    $T_1$ is the only class of $\RR$ which may contain more than one
    vertex of $f_\gamma$).

    By Lemma~\ref{l:subst}(i)--(ii), together with the fact that each
    $Y\in\RR$ is $\rho_X$-solid (where $Y\sub X\sub\PP_0^\pi$), each
    $T_i$ ($i=1,\dots,k$) is $\rho$-solid. Thus, $T_i$ is also
    $\sigma$-solid for $i\geq 2$. Let $T'_1$ be the negative
    $\sigma[T_1]$-part of $\inner H \sigma {T_1}$ containing $u_1$.

    For $i=1,\dots,k-1$, let $e_i$ be the hyperedge of $E(\gamma)$
    such that $\gamma(e_i)/\RR = T_iT_{i+1}$ (choosing $e_1\neq
    f_\gamma$ if $k=2$). Let $T = T'_1\cup T_2\cup\dots\cup
    T_k$. Using the minimality of $H$ and Lemma~\ref{l:sub}(ii), it is
    easy to prove that $T$ is a subset of a class, say $Q$, of
    $\sdneg\sigma$. Note that $Q$ contains $u_1$ and $u_2$ as claimed.

    If \eqref{eq:sdneg-same} is false, then the unique vertex of
    $f_\gamma-\Setx{u_1,u_2}$ is necessarily contained in a class of
    $\sdneg\sigma$ distinct from $Q$. In that case, however,
    $\sdneg\sigma$ is not $\sigma$-narrow as $\sigma(f_\gamma) \sub
    Q$. This contradiction with the definition proves
    \eqref{eq:sdneg-same}.

    By \eqref{eq:strictly} and \eqref{eq:sdneg-same}, $\pi \preceq
    \rho \prec \sigma$. Moreover, $\sigma$ is acyclic, since $\rho$ is
    acyclic and $\sigma(f_\gamma)$ has endvertices in distinct
    components of $\rho^*$. Thus, $\sigma$ satisfies condition (a) in
    the statement of the lemma, contradicting the choice of $\pi$.
  \end{claimproof}

  For any $Y\in\RR$, let $\conn Y$ be the set of all connectors for
  $Y$, and write
  \begin{equation*}
    \iconn Y = \Set{f\cap Y}{f\in\conn Y}.
  \end{equation*}
  Note that for any connector $f$ for $Y$, $f\cap Y$ is a 2-hyperedge
  of $\rho[Y]$.

  Let us describe our strategy in the next step in intuitive terms. We
  want to modify $\rho$ within the classes of $\RR$ and `free' one of
  the hyperedges $f_\gamma$ from $\rho$, which would enable us to
  apply the argument from the proof of Claim~\ref{cl:all-used} and
  reach a contradiction. If no such modification works, we obtain a
  quasigraph $\sigma$ and a partition $\SS$ which refines $\RR$. The
  effect of the refinement is to `destroy' all quasicycles
  $\gamma/\RR$ in $\barr{\rho/\RR}$ by making the representation
  $\rho(f_\gamma)$ of each associated connector $f_\gamma$
  $\SS$-crossing. Thanks to this, it will turn out that $\SS$ is
  $\sigma$-skeletal as required to satisfy condition (b).
  
  Thus, let $Y\in\RR$ and set
  \begin{align*}
    \tilde H_Y &= \inner H \rho Y - \iconn Y,\\
    \tilde\rho_Y &= \rho[Y] - \iconn Y
  \end{align*}
  (we allow $\iconn Y = \emptyset$) and observe that $\tilde\rho_Y$ is
  an acyclic quasigraph in $\tilde H_Y$. Let $\sigma_Y$ be a
  $\preceq$-maximal acyclic quasigraph in $\tilde H_Y$ such that
  $\sigma_Y \succeq \tilde\rho_Y$. We define a quasigraph $\tau_Y$
  in $\inner H \rho Y$ by
  \begin{equation*}
    \tau_Y(e) =
    \begin{cases}
      e & \text{if $e\in\iconn Y$,}\\
      \sigma_Y(e) & \text{otherwise.}
    \end{cases}
  \end{equation*}

  \begin{claim}
    \label{cl:strictly-conn}
    For all $Y\in\RR$,
  \begin{equation*}
    \sdposs{\sigma_Y}{\tilde H_Y} = \sdposs{\tilde\rho_Y}{\tilde H_Y}.
  \end{equation*}
  \end{claim}
  \begin{claimproof}
    From $\sigma_Y \succeq \tilde\rho_Y$, we know that the left
    hand side in the statement of the claim is coarser than (or equal
    to) the right hand side. Suppose that for some $Y\in\RR$,
    $\sdposs{\sigma_Y}{\tilde H_Y}$ is strictly coarser than
    $\sdposs{\tilde\rho_Y}{\tilde H_Y}$. Then we can choose vertices
    $u_1,u_2\in Y$ which are contained in different classes $U_1,U_2$,
    respectively, of $\sdposs{\tilde\rho_Y}{\tilde H_Y}$, but in the
    same class $U$ of $\sdposs{\sigma_Y}{\tilde H_Y}$. Since $Y$
    is $\rho$-solid, the graph $\rho[Y]^*$ contains a path $P$ joining
    $u_1$ to $u_2$. The choice of $u_1$ and $u_2$ implies the
    following:
    \begin{enumerate}
    \item[\quad(A1)] $P$ contains the edge $f_\gamma\cap Y \in \iconn
      Y$ for some quasicycle $\gamma$, and
    \item[\quad(A2)] all the edges of $E(P) \cap \iconn Y$ are
      contained in a cycle in $(\subst\rho{\sigma_Y})^*$.
    \end{enumerate}
    We choose a quasicycle $\gamma$ satisfying (A1) and let $\tau$
    be the quasigraph in $H$ obtained as
    \begin{equation*}
      \tau = (\subst\rho{\tau_Y}) - f_\gamma\cap Y.
    \end{equation*}

    \begin{figure}
      \centering \sfig{19}{Bold lines show the quasigraph
        $\rho$, the dotted regions are the positive
        $\tilde\rho_Y$-parts of $\tilde H_Y$.}\hfilll
      \sfig{20}{The dotted regions here are the positive
        $\sigma_Y$-parts of $\tilde H_Y$. If the partition is strictly
        coarser than in (a), we can `free' a suitable
        connector $f_\gamma$ and use it as before.}\\
      \hfilll\sfig{21}{Otherwise, we obtain a finer partition $\SS$
        (darkest grey regions) such that $\rho(f_\gamma)$ is
        $\SS$-crossing for each $\gamma$.}\hfilll
      \caption{An illustration to the proof of
        Claim~\ref{cl:strictly-conn} and the following part of the
        proof. We use similar conventions as in
        Figure~\ref{fig:skeletal1}.}
      \label{fig:skeletal2}
    \end{figure}

    By (A2) and the fact that $\rho[Y]$ is connected, $\tau[Y]$ is
    connected as well. Since $\sigma_Y$ has tight complement in
    $\tilde H_Y$, $\tau[Y]$ has tight complement in $\inner H \rho Y$
    (the two complements coincide). Thus, $Y$ is $\tau$-solid. By
    Corollary~\ref{cor:long}, $\tau\succeq\rho$. By
    Lemma~\ref{l:break-cycles} and the fact that $\rho\succeq\pi$, we
    may assume that $\tau$ is acyclic.

    Since $\rho$ and $\tau$ are $Y$-similar, we have
    \begin{equation*}
      \barr{\rho/\RR} = \barr{\tau/\RR}.
    \end{equation*}
    In particular, the quasicycle $\gamma$ in $\barr{\rho/\RR}$
    (associated with $f_\gamma$) is also a quasicycle in
    $\barr{\tau/\RR}$. As $f_\gamma$ is not used by $\tau$ (and
    $\tau \succeq \rho$), we can repeat the argument used in the
    proof of Claim~\ref{cl:all-used}, namely add $f_\gamma$ (with a
    suitable representation) to $\tau$ and reach a contradiction
    with the choice of $\pi$.
  \end{claimproof}

  We will now construct a $\sigma$-skeletal partition of $V$. Let
  $Y\in\RR$. By the choice of $H$ and the maximality of $\sigma_Y$,
  there is a $\sigma_Y$-skeletal partition $\SS_Y$ of $Y$ (in $\tilde
  H_Y$). We define a quasigraph $\sigma$ in $H$ and a partition $\SS$
  of $V$ by
  \begin{align*}
    \sigma &= \subst \rho {\Set{\tau_Y}{Y\in\RR}},\\
    \SS &= \bigcup_{Y\in\RR} \SS_Y.
  \end{align*}

  We aim to show that $\SS$ is $\sigma$-skeletal. Let $Z\in\SS$ and
  suppose that $Z \sub Y \sub X$, where $X\in\PP_0^\pi$ and $Y\in
  \RR$. Since $\sigma[Z] = \sigma_Y[Z]$ and $\SS_Y$ is
  $\sigma_Y$-skeletal, $\sigma[Z]$ is a quasitree. 

  To show that the complement of $\sigma[Z]$ in $\inner H \sigma Z$ is
  tight, we use Lemma~\ref{l:subst}(i):
  \begin{equation}\label{eq:innerZ}
    \inner H \sigma Z = \inner {(\inner H \rho Y)} {\tau_Y} Z
    = \inner{\tilde H_Y}{\tau_Y} Z
    = \inner{\tilde H_Y}{\sigma_Y} Z.
  \end{equation}
  Here, the second and the third equality follows from
  Claim~\ref{cl:strictly-conn} which implies that any connector for
  $Y$ intersects two classes of $\sdposs{\sigma_Y}{\tilde
    H_Y}$. From~\eqref{eq:innerZ} and the fact that $\sigma_Y[Z]$ has
  tight complement in $\inner{\tilde H_Y}{\sigma_Y} Z$, it follows
  that $\sigma[Z]$ has tight complement as well.

  It remains to prove that $\barr{\sigma/\SS}$ is acyclic. Suppose,
  for the sake of a contradiction, that $\gamma$ is a quasigraph in
  $H$ such that $\gamma/\SS$ is a quasicycle in $\barr{\sigma/\SS}$.
  Note that the complement of $\tau_Y/\SS_Y$ in $\inner H \rho Y$
  is the same as the complement of $\sigma_Y/\SS_Y$ in $\tilde H_Y$,
  and hence acyclic. By Lemma~\ref{l:cycle-di}, $\gamma/\RR$ is a
  nonempty quasigraph in $\barr{\rho/\RR}$ with $(\gamma/\RR)^*$
  eulerian. 

  Let $\delta$ be a restriction of $\gamma$ such that $\delta/\RR$ is
  a quasicycle in $\barr{\rho/\RR}$. Every such quasicycle has an
  associated hyperedge $f_{\delta}$ which is a connector for a class
  $Y_\delta\in\RR$ (Claim~\ref{cl:all-used}). In particular,
  $f_\delta$ is used by $\rho$. By the fact that $f_\delta$ intersects
  two classes of $\sdposs{\sigma_{Y_\delta}}{\tilde H_{Y_\delta}}$,
  $\rho(f_\delta)$ is $\SS$-crossing. This implies that
  $\sigma(f_\delta)$ is $\SS$-crossing, which contradicts the
  assumption that $\gamma/\SS$ is a quasicycle in
  $\barr{\sigma/\SS}$. The proof is complete.
\end{proof}


\section{Proof of Theorem~\ref{t:spanning-tree}}
\label{sec:proof-main}

We can now prove our main result regarding spanning trees in
hypergraphs, announced in Section~\ref{sec:quasi} as
Theorem~\ref{t:spanning-tree}:

\newtheorem*{thm-tree}{Theorem}
\begin{thm-tree}
  Let $H$ be a 4-edge-connected 3-hypergraph. If no 3-hyperedge of $H$
  is included in any edge-cut of size 4, then $H$ contains a quasitree
  with tight complement.
\end{thm-tree}
\begin{proof}
  Let $\pi$ be a $\preceq$-maximal acyclic quasigraph in $H$. By the
  Skeletal Lemma (Lemma~\ref{l:skeletal}), there exists a
  $\pi$-skeletal partition $\PP$ of $V$. For the sake of a
  contradiction, suppose that $\pi$ is not a quasitree with tight
  complement. In particular, $\PP$ is nontrivial.

  Assume that $H/\PP$ has $n$ vertices (that is, $\size\PP = n$) and
  $m$ hyperedges. For $k\in\Setx{2,3}$, let $m_k$ be the number of
  $k$-hyperedges of $\pi/\PP$. Similarly, let $\barr{m_k}$ be the
  number of $k$-hyperedges of $\barr{\pi/\PP}$. Thus, $m = m_2 + m_3 +
  \barr{m_2} + \barr{m_3}$.

  Since $\barr{\pi/\PP}$ is acyclic, the graph $\gr{\barr{\pi/\PP}}$
  (defined in Section~\ref{sec:prelim}) is a forest. As
  $\gr{\barr{\pi/\PP}}$ has $n+\barr{m_3}$ vertices and
  $\barr{m_2}+3\barr{m_3}$ edges, we find that
  \begin{equation}
    \label{eq:compl}
    \barr{m_2} + 2\barr{m_3} \leq n-1.
  \end{equation}
  Since $\PP$ is $\pi$-solid and $\pi$ is an acyclic quasigraph, we
  know that $m_2 + m_3 \leq n-1$. Moreover, by the assumption that
  $\pi$ is not a quasitree with a tight complement, either this
  inequality or \eqref{eq:compl} is strict. Summing the two, we obtain
  \begin{equation}
    \label{eq:both}
    m + \barr{m_3} \leq 2n-3.
  \end{equation}
    
  We let $n_4$ be the number of vertices of $H/\PP$ of degree 4, and
  $n_{5^+}$ be the number of the other vertices. Since $n \geq 2$ and
  $H$ is 4-edge-connected, we have $n = n_4 + n_{5^+}$. By double
  counting,
  \begin{equation}
    \label{eq:verts}
    4n_4 + 5n_{5^+} \leq 2(m_2+\barr{m_2}) + 3(m_3+\barr{m_3}) = 2m +
    m_3 + \barr{m_3}.
  \end{equation}
  The left hand side equals $4n + n_{5^+}$. Using~\eqref{eq:both},
  we find that 
  \begin{equation*}
    4n + n_{5^+} \geq 2m + 2\barr{m_3} + n_{5^+} + 6.
  \end{equation*}
  Combining with~\eqref{eq:verts}, we obtain
  \begin{equation}\label{eq:star}
    m_3 \geq \barr{m_3} + n_{5^+} + 6.
  \end{equation}

  We show that $m_3 \leq n_{5^+}$. Let $T' = (\pi/\PP)^*$ be the
  forest on $\PP$ which represents $\pi/\PP$. In each component of
  $T'$, choose a root and direct the edges of $T'$ away from it. To
  each 3-hyperedge $e\in E(\pi/\PP)$, assign the head $h(e)$ of the
  arc $\pi(e)$. By the assuptions of the theorem, no edge-cut of size
  4 contains a 3-hyperedge, so $h(e)$ is a vertex of degree at least
  5. At the same time, since each vertex is the head of at most one
  arc in the directed forest, it gets assigned to at most one
  hyperedge. The inequality $m_3 \leq n_{5^+}$ follows.  This
  contradiction to inequality~\eqref{eq:star} proves that $\pi$ is
  a quasitree with tight complement.
\end{proof}


\section{Even quasitrees}
\label{sec:even}

In the preceding sections, we were busy looking for quasitrees with
tight complement in hypergraphs. In this and the following section, we
will explain the significance of such quasitrees for the task of
finding a Hamilton cycle in the line graph of a given graph. 

Let $\pi$ be a quasitree in $H$. For a set $X\sub V$, we define a
number $\Phi_\pi(X) \in \Setx{0,1}$ by
\begin{equation*}
  \Phi_\pi(X) \equiv \sum_{v\in X} d_{\pi^*}(v) \pmod 2.
\end{equation*}
Observe that $\Phi_\pi(X) = 0$ if and only if $X$ contains an even
number of vertices whose degree in the tree $\pi^*$ is odd.

For $X\sub V$, we say that $\pi$ is \emph{even on $X$} if for every
component $K$ of $\barr\pi$ whose vertex set is a subset of $X$, it
holds that $\Phi_\pi(V(K)) = 0$. If $\pi$ is even on $V$, then we just
say $\pi$ is \emph{even}.

The main result of this section is the following:
\begin{lemma}\label{l:ext}
  If $\pi$ is a quasitree in $H$ with tight complement, then there is
  a quasigraph $\rho$ in $H$ such that $E(\rho) = E(\pi)$ and $\rho$
  is an even quasitree in $H$.
\end{lemma}

Lemma~\ref{l:ext} is a direct consequence of the following more
technical statement (to derive Lemma~\ref{l:ext}, set $X=V$):

\begin{lemma}
  \label{l:ext1}
  Let $\pi$ be a quasitree in $H$ and $X\sub V$. Assume that
  $\Phi_\pi(X) = 0$ and $\pi$ has tight complement in $\inner H \pi
  X$. Then there is a quasitree $\rho$ in $H$ such that $\pi$ and
  $\rho$ are $X$-similar, and $\rho$ is even on $X$.
\end{lemma}
\begin{proof}
  We proceed by induction on $\size X$. We may assume that $\size X
  \geq 2$, since otherwise the claim is trivially true. Similarly, if
  $\barr\pi[X]$ is connected, then the assumption $\Phi_\pi(X) = 0$
  implies that $\pi$ is even on $X$. Thus, we assume that
  $\barr\pi[X]$ is disconnected.

  The definition implies that there is a partition $X = X_1\cup X_2$
  such that:
  \begin{enumerate}[\quad(B1)]
  \item for each $i=1,2$, $\pi[X_i]$ has tight complement in $\inner H
    \pi {X_i}$, 
  \item there is a hyperedge $e$ intersecting $X_2$ with $\pi(e) \sub
    X_1$, and
  \item for any hyperedge $f$ intersecting both $X_1$ and $X_2$, we
    have $f\in E(\pi)$.
  \end{enumerate}
  
  If $\Phi_\pi(X_1) = 0$, then we may use the induction hypothesis
  with $X_1$ playing the role of $X$. The result is a quasitree
  $\rho_1$ in $H$ which is even on $X_1$ and $X_1$-similar to
  $\rho$. In particular, $\Phi_{\rho_1}(X_1) = 0$ and hence also
  $\Phi_{\rho_1}(X_2) = 0$. Using the induction hypothesis for $X_2$,
  we obtain a quasitree $\rho_2$ in $H$ which is even on $X_2$;
  furthermore, being $X_2$-similar to $\rho_1$, it is even on $X_1$ as
  well. By (B3), the vertex set of every component $K$ of $\barr\pi$
  with $V(K)\sub X$ is a subset of $X_1$ or $X_2$. Thus, $\rho :=
  \rho_2$ is even on $X$, and clearly $X$-similar to $\pi$.

  \begin{figure}
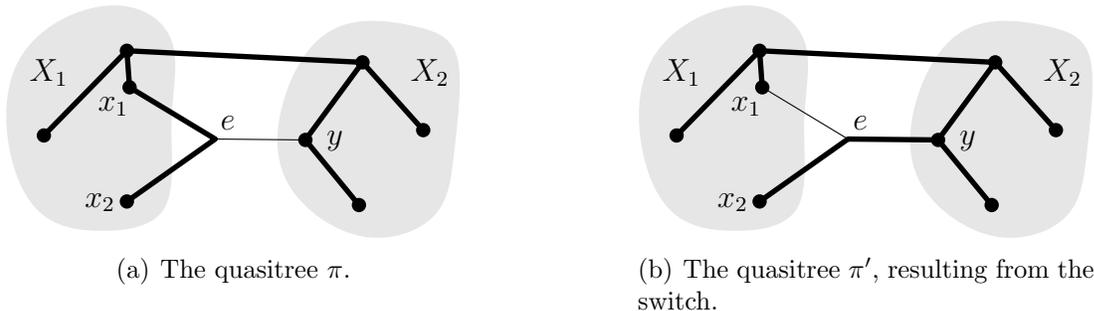

    \centering
    \sfig{22}{The quasitree $\pi$.}
    \hfilll
    \sfig{23}{The quasitree $\pi'$, resulting from the switch.}
    \caption{The case $\Phi_\pi(X_1)=1$ in the proof of
      Lemma~\ref{l:ext1}. The gray regions are the sets $X_1$ and
      $X_2$. Note how the switch of the representation of $e$ changes
      the parity of exactly one vertex degree in $X_1$.}
    \label{fig:odd}
  \end{figure}

  It remains to consider the case that $\Phi_\pi(X_1) = 1$,
  illustrated in Figure~\ref{fig:odd}. Here we
  need to `switch' the representation of $e$ (the hyperedge from (B2))
  as follows. Let $e=x_1x_2y$, with $\pi(e) = x_1x_2$. The removal of
  the edge $x_1x_2$ from $\pi^*$ splits $\pi^*$ into two components,
  each containing one of $x_1$ and $x_2$. By symmetry, we may assume
  that $y$ is contained in the component containing $x_1$. We define a
  new quasigraph $\pi'$ in $H$ by
  \begin{equation*}
    \pi'(e) =
    \begin{cases}
      x_2y & \text{ if $f=e$,}\\
      \pi(f) & \text{otherwise.}
    \end{cases}
  \end{equation*}
  Note that $\pi'$ is a quasitree and $\Phi_{\pi'}(X_1) =
  0$. Consequently, we can proceed as before, apply the induction
  hypothesis and eventually obtain a representation $\rho$ which
  satisfies the assertions of the lemma.
\end{proof}


\section{Hamilton cycles in line graphs and claw-free graphs}
\label{sec:claw-free}

We recall two standard results which interpret the connectivity and
the hamiltonicity of a line graph in terms of its preimage. The first
result is a folklore observation, the second is due to Harary and
Nash-Williams~\cite{bib:HNW-eulerian}. We combine them into one
theorem, but before we state them, we recall some necessary
terminology.

Let $G$ be a graph. An edge-cut $C$ in $G$ is \emph{trivial} if it
consists of all the edges incident with some vertex $v$ of $G$. The
graph $G$ is \emph{essentially $k$-edge-connected} ($k\geq 1$) if
every edge-cut in $G$ of size less than $k$ is trivial. A subgraph $D$
of $G$ is \emph{dominating} if $G-V(D)$ has no edges.

\begin{theorem}\label{t:line-basic}
  For any graph $G$ and $k\geq 1$, the following holds:
  \begin{enumerate}[\quad(i)]
  \item $L(G)$ is $k$-connected if and only if $G$ is essentially
    $k$-edge-connected,
  \item $L(G)$ is hamiltonian if and only if $G$ contains a dominating
    connected eulerian subgraph $C$.
  \end{enumerate}
\end{theorem}

In a similar spirit, the minimum degree of $L(G)$ equals the minimum
edge weight of $G$, where the \emph{weight} of an edge $e$ is defined
as the number of edges incident with $e$ and distinct from it.

Given a set $X$ of vertices of $G$, an \emph{$X$-join} in $G$ is a
subgraph $G'$ of $G$ such that a vertex of $G$ is in $X$ if and only
if its degree in $G'$ is odd. (In particular, $\emptyset$-joins are
eulerian subgraphs.) 

We will need a lemma which has been used a number of times before,
either explicitly or implicitly. For completeness, we sketch a quick
proof.

\begin{lemma}\label{l:join}
  If $T$ is a tree and $X$ is a set of vertices of $T$ of even
  cardinality, then $T$ contains an $X$-join. 
\end{lemma}
\begin{proof}
  By induction on the order of $T$. If $\size{V(T)} = 1$, the
  assertion is trivial. Otherwise, choose an edge $e = v_1v_2$ and let
  $T_1$ and $T_2$ be components of $T-e$, $T_1$ being the one which
  contains $v_1$. Let $X_1$ be $X\cap V(T_1)$ if the size of this set
  is even; otherwise, set $X_1 = (X\cap V(T_1)) \oplus \Setx{v_1}$,
  where $\oplus$ stands for the symmetric difference. The induction
  yields an $X_1$-join $T'_1$ in $T_1$. A set $X_2$ and an $X_2$-join
  $T'_2$ in $T_2$ is obtained in a symmetric way. It is easy to check
  that the union of $T'_1$ and $T'_2$, with $e$ added if $\size{X\cap
    V(T_1)}$ is odd, is an $X$-join.
\end{proof}

If $G_1$ and $G_2$ are two graphs, then $G_1+G_2$ denotes the graph
whose vertex set is the (not necessarily disjoint) union of vertex
sets of $G_1$ and $G_2$, and whose multiset of edges is the multiset
union of $E(G_1)$ and $E(G_2)$. 

As the following lemma shows, an even quasitree in $H$ allows one to
find a connected spanning eulerian subgraph of $\gr H$ (see
Figure~\ref{fig:odd2} for an illustration):

  \begin{figure}
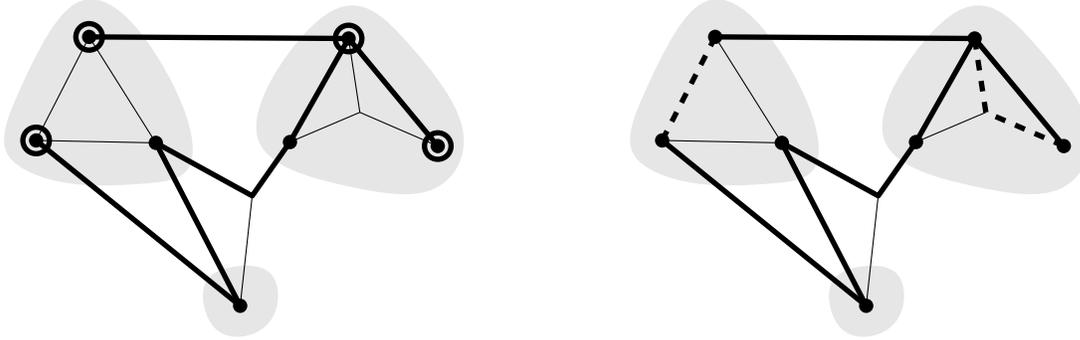

    \centering
    \sfig{24}{The vertices of odd degree in $\pi^*$ are circled.}
    \hfilll
    \sfig{25}{Since each component of $\barr\pi$ contains an even number of
      circled vertices, we can complete $\pi^*$ to an eulerian graph
      (the added edges are shown as dashed bold lines).}
    \caption{An illustration to Lemma~\ref{l:eulerian}. The gray
      regions are the components of $\barr\pi$, where $\pi$ is the
      quasigraph shown by solid bold lines.}
    \label{fig:odd2}
  \end{figure}

\begin{lemma}\label{l:eulerian}
  If $\pi$ is an even quasitree in $H$, then there is a quasigraph
  $\tau$ in $H$ such that $E(\pi)$ and $E(\tau)$ are disjoint, and
  $\pi^* + \tau^*$ is a connected eulerian subgraph of the graph
  $\gr H$ spanning all vertices in $V$.
\end{lemma}
\begin{proof}
  Let $K$ be a component of $\barr\pi$, and let $X$ be the set of
  vertices of $K$ whose degree in $\pi^*$ is odd. Since $\pi$ is even,
  $\size X$ is even. Choose a spanning tree $T$ of the (connected)
  graph $\gr K$. Using Lemma~\ref{l:join}, choose a subforest $T'$ of
  $T$ such that for every vertex $w$ of $\gr K$, $d_{T'}(w)$ is odd if
  and only if $w\in X$. In $\pi^* + T'$, all the vertices of $K$ have
  even degrees. In fact, the same holds for any vertex $v_e$ of
  $\gr K$, where $e$ is a hyperedge of $H$ of size 3: if $e$ is used by
  $\pi$, then $d_{\pi^* + T'}(v_e) = 2$, and otherwise we have
  \begin{equation*}
    d_{\pi^* + T'}(v_e) = d_{T'}(v_e),
  \end{equation*}
  which is even since $v_e\notin X$. In particular, there is a
  quasigraph $\tau_K$ in $H$ such that $\tau_K^* = T'$.

  We apply the above procedure repeatedly, one component of $\barr\pi$
  at a time. For this, we need to be sure that a 3-hyperedge $e$ will
  not be used by $\tau_{K_1}$ as well as $\tau_{K_2}$, where $K_1$ and
  $K_2$ are distinct components of $\barr\pi$. This is clear, however,
  since $e$ can only be used by $\tau_K$ if $\size{e\cap V(K)}\geq
  2$. Thus, the components of $\barr\pi$ can be treated independently,
  and we eventually obtain an eulerian subgraph $S$ of $\gr H$. Since
  it contains the tree $\pi^*$, $S$ spans all of $V$, and since each
  of the trees $(\tau_K)^*$ contains an edge incident with a vertex in
  $V$ (unless $(\tau_K)^*$ is edgeless), it follows that $S$ is
  connected.
\end{proof}

Using Theorem~\ref{t:line-basic}, it will be easy to derive our main
result (Theorem~\ref{t:main}) as a consequence of the following
proposition. Let us remark that the proposition is closely related to
a conjecture made by B. Jackson
(see~\cite[Conjecture~4.48]{bib:BBHV-progress}) and implies one of its
three versions.

\begin{proposition}\label{p:essentially}
  If $G$ is an essentially 5-edge-connected graph with minimum edge
  weight at least 6, then $G$ contains a connected eulerian subgraph
  spanning all the vertices of degree at least 4 in $G$.
\end{proposition}
\begin{proof}
  For the sake of a contradiction, let $G$ be a counterexample with as
  few vertices as possible. Since the claim is trivially true for a
  one-vertex graph, we may assume $\size{V(G)} \geq 2$. For brevity, a
  \emph{good subgraph} in a graph $G'$ will be a connected eulerian
  subgraph spanning all the vertices of degree at least 4 in $G'$.
  \setcounter{claim}{0}
  \begin{claim}\label{cl:mindeg}
    The minimum degree of $G$ is at least 3.
  \end{claim}
  \begin{claimproof}
    Suppose first that $G$ contains a vertex $v$ of degree 2 with
    distinct neighbours $w_1$ and $w_2$. If we suppress $v$, the
    resulting graph $G'$ will be essentially
    5-edge-connected. Furthermore, the minimum edge weight of $G'$ is
    at least 6 unless $G$ is the triangle $vw_1w_2$ with the edge
    $w_1w_2$ of multiplicity 5, which is however not a counterexample
    to the proposition. By the minimality assumption, $G'$ contains a
    good subgraph $C'$. It is easy to see that the corresponding
    subgraph of $G$ is also good.

    Suppose then that $G$ contains a vertex $u$ of degree 1 or 2 with
    a single neighbour $z$. Let $U$ be the set of all the vertices of
    degree 1 or 2 in $G$ whose only neighbour is $z$. If $V(G) = U
    \cup \Setx z$, then the Eulerian subgraph consisting of just the
    vertex $z$ shows that $G$ is not a counterexample to the
    proposition. Thus, $z$ has a neighbour $x$ outside $U$. In fact,
    since $G$ is essentially 5-edge-connected, $z$ is incident with at
    least 5 edges whose other endvertex is not in $U$. Let $e$ be an
    edge with endvertices $z$ and $x$. Since the degree of $x$ is at
    least 3, the edge weight of $e$ in $G-U$ is at least 6. This
    implies that the minimum edge weight of $G-U$ is at least 6. Since
    the removal of $U$ does not create any new minimal essential
    edge-cut, $G-U$ is essentially 5-edge-connected. Since the degree
    of $z$ in $G-U$ is at least 5, any good subgraph in $G-U$ is a
    good subgraph in $G$. Thus, $G-U$ is a smaller counterexample than
    $G$, contradicting the minimality of $G$.
  \end{claimproof}
  
  \begin{claim}\label{cl:parallel}
    No vertex of degree 3 in $G$ is incident with a pair of parallel
    edges.
  \end{claim}
  \begin{claimproof}
    Suppose that $v$ is a vertex of degree 3 incident with parallel
    edges $e_1,e_2$. If $v$ has only one neighbour, then any good
    subgraph of $G-v$ is good in $G$. By the minimality of $G$, $v$
    must have exactly two neighbours, say $w$ and $z$, where $w$ is
    incident with $e_1$ and $e_2$. Let $G'$ be obtained from $G$ by
    removing $v$ and adding the edge $e_0$ with endvertices $w$ and
    $z$.

    It is easy to see that $G'$ is essentially 5-edge-connected, and
    that any good subgraph of $G'$ can be modified to a good subgraph
    of $G$ (as $d_G(w)\geq 6$). We show that the minimal edge weight
    in $G'$ is at least 6.

    Suppose the contrary and let $e$ be an edge of $G'$ of weight less
    than 6. We have $e\neq e_0$ as the assumptions imply that
    $d_G(w)\geq 6$ and $d_G(z)\geq 5$, so the weight of any edge with
    endvertices $w$ and $z$ in $G'$ is at least 8. Thus, $e$ is an
    edge of $G$.
    
    It must be incident with $w$, for otherwise its weight in $G'$
    would be the same as in $G$. Let $u$ be the endvertex of $e$
    distinct from $w$. Since $d_G(w)\geq 6$, $w$ is incident in $G'$
    with at least 3 edges of $G'$ distinct from $e_0$ and $e$. By the
    weight assumption, $u$ must be incident with only at most one edge
    of $G'$ other than $e$, contradicting Claim~\ref{cl:mindeg}. 
  \end{claimproof}

  Let $H$ be the 3-hypergraph whose vertex set $V$ is the set of all
  vertices of $G$ whose degree is at least 4; the hyperedges of
  $H$ are of two kinds:
  \begin{itemize}
  \item the edges of $G$ with both endvertices in $V$,
  \item 3-hyperedges consisting of the neighbours of any vertex of
    degree 3 in $G$. 
  \end{itemize}
  Note that $H$ is well-defined, for any neighbour of a vertex of
  degree 3 in $G$ must have degree at least 4 (otherwise they would be
  separated from the rest of the graph by an essential edge-cut of
  size at most 4). Furthermore, by Claim~\ref{cl:parallel}, any vertex
  of degree 3 does indeed have three distinct neighbours in $V$. 

  In the following two claims, we show that $H$ satisfies the
  hypotheses of Theorem~\ref{t:spanning-tree}.

  \begin{claim}\label{cl:4-conn}
    The hypergraph $H$ is 4-edge-connected.
  \end{claim}
  \begin{claimproof}
    Suppose that this is not the case and $F$ is an inclusionwise
    minimal edge-cut in $H$ with $\size F \leq 3$. Let $A$ be the
    vertex set of a component of $H-F$.

    Let $e\in F$. By the minimality of $G$, $\size{e-A}\geq 1$. We
    assign to $e$ an edge $e'$ of $G$, defined as follows:
    \begin{itemize}
    \item if $\size e = 2$, then $e' = e$,
    \item if $\size e = 3$ and $e\cap A = \Setx{u}$, then $e' = uv_e$,
    \item if $\size e = 3$, $\size{e\cap A} = 2$ and $e-A = \Setx{u}$,
      then $e' = uv_e$.
    \end{itemize}
    Observe that $F' := \Set{e'}{e\in F}$ is an edge-cut in $G$. Since
    $G$ is 5-edge-connected, $F'$ must be a trivial edge-cut. This
    means that a vertex $v\in V$ has degree 3 in $H$, a contradiction
    as $v$ has degree at least 4 in $G$ and therefore also in $H$.
  \end{claimproof}
  
  The other claim regards edge-cuts of size 4 in $H$:

  \begin{claim}
    No 3-hyperedge of $H$ is included in an edge-cut of size 4 in
    $H$. 
  \end{claim}
  \begin{claimproof}
    Let $F$ be an edge-cut of size 4 in $H$. As in the proof of
    Claim~\ref{cl:4-conn}, we consider the corresponding edge-cut
    $F'$ in $G$. Since $G$ is essentially 5-edge-connected, one
    component of $G-F'$ consists of a single vertex $w$ whose degree
    in $G$ is 4. Assuming that $F$ includes a 3-hyperedge $e$, we find
    that in $G$, $w$ has a neighbour $v$ of degree 3. Since the weight
    of the edge $vw$ is 5, we obtain a contradiction with our
    assumptions about $G$.
  \end{claimproof}

  Since the assumptions of Theorem~\ref{t:spanning-tree} are
  satisfied, we can use it to find a quasitree $\pi$ with tight
  complement in $H$. By Lemmas~\ref{l:ext} and \ref{l:eulerian}, $\gr H
  = G$ admits a connected eulerian subgraph spanning the set $V$. This
  is what we wanted to find.
\end{proof}

We can now prove our main theorem, stated as Theorem~\ref{t:main} in
Section 1: \newtheorem*{thm}{Theorem}
\begin{thm}
  Every 5-connected line graph of minimum degree at least 6 is
  hamiltonian.
\end{thm}
\begin{proof}
  Let $L(G)$ be a 5-connected line graph of minimum degree at least
  6. By Theorem~\ref{t:line-basic}(i), $G$ is essentially
  5-edge-connected. Furthermore, the minimum edge weight of $G$ is at
  least 6. By Proposition~\ref{p:essentially}, $G$ contains a
  connected eulerian subgraph $C$ spanning all the vertices of degree
  at least 4. By Theorem~\ref{t:line-basic}(ii), it is sufficient to
  prove that $G-V(C)$ has no edges. Indeed, the vertices of any edge
  $e$ in $G-V(C)$ must have degree at most 3 in $G$, which implies
  that $e$ is incident to at most 4 other edges of $G$, a
  contradiction to the minimum degree assumption. Thus, $L(G)$ is
  hamiltonian.
\end{proof}

Using the claw-free closure concept developed by
Ryj\'{a}\v{c}ek~\cite{bib:Ryj-closure}, Theorem~\ref{t:main} can be
extended to claw-free graphs. Let us recall the main result of
\cite{bib:Ryj-closure}:
\begin{theorem}\label{t:closure}
  Let $G$ be a claw-free graph. Then there is a well-defined graph
  $cl(G)$ (called the \emph{closure} of $G$) such that the following
  holds:
  \begin{enumerate}[\quad (i)]
  \item $G$ is a spanning subgraph of $cl(G)$,
  \item $cl(G)$ is the line graph of a triangle-free graph,
  \item the length of a longest cycle in $G$ is the same as in
    $cl(G)$.
  \end{enumerate}
\end{theorem}

\begin{corollary}\label{cor:claw-free}
  Every 5-connected claw-free graph $G$ of minimum degree at least 6
  is hamiltonian.
\end{corollary}
\begin{proof}
  Apply Theorem~\ref{t:closure} to obtain the closure $cl(G)$ of
  $G$. Since $G \sub cl(G)$, the closure is 5-connected and has
  minimum degree at least 6. Being a line graph, $cl(G)$ is
  hamiltonian by Theorem~\ref{t:main}. Since $G$ is a spanning
  subgraph of $cl(G)$, property (iii) in Theorem~\ref{t:closure}
  implies that $G$ is hamiltonian.
\end{proof}


\section{Hamilton-connectedness}
\label{sec:ham-conn}

Recall from Section~\ref{sec:intro} that a graph is Hamilton-connected
if for every pair of distinct vertices $u,v$, there is a Hamilton path
from $u$ to $v$. The method used to prove Theorem~\ref{t:main} and
Corollary~\ref{cor:claw-free} can be adapted to yield the following
stronger result:

\begin{theorem}
  \label{t:ham-conn}
  Every 5-connected claw-free graph of minimum degree at least 6 is
  Hamilton-connected.
\end{theorem}

In this section, we sketch the necessary modifications to the
argument. For a start, let $H = L(G)$ be a 5-connected line graph of
minimum degree at least 6. By considerations similar to those in the
proof of Proposition~\ref{p:essentially}, it may be assumed that the
minimum degree of $G$ is at least 3 and that no vertex of $G$ is
incident with a pair of parallel edges, so we may associate with $G$ a
3-hypergraph $H$ just as in that proof. Moreover, $H$ may again be
assumed to satisfy the assumptions of Theorem~\ref{t:spanning-tree}.

Let $V_{\geq 4} \subseteq V(G)$ be the set of vertices of degree at
least 4 in $G$.

First, we will need a replacement of Theorem~\ref{t:line-basic}(ii)
that translates the Hamilton-connectedness of $H$ to a property of
$G$. A \emph{trail} $F$ is a sequence of edges of $G$ such that each
pair of consecutive edges is adjacent in $G$, and $F$ contains each
edge of $G$ at most once. We will say that $F$ \emph{spans} a set $Y$
of vertices if each vertex in $Y$ is incident with an edge of $F$. A
trail is an \emph{$(e_1,e_2)$-trail} if it starts with $e_1$ and ends
with $e_2$. Furthermore, an $(e_1,e_2)$-trail $F$ is \emph{internally
  dominating} if every edge of $G$ has a common endvertex with some
edge in $F$ other than $e_1$ and $e_2$. The following fact is
well-known (see, e.g., \cite{bib:LLZ-eulerian}):

\begin{theorem}\label{t:ham-conn-preimage}
  Let $G$ be a graph with at least 3 edges. Then $L(G)$ is
  Hamilton-connected if and only if for any pair of edges $e_1,e_2\in
  E(G)$, $G$ has an internally dominating $(e_1,e_2)$-trail.
\end{theorem}

One way to find an internally dominating $(e_1,e_2)$-trail (where
$e_1,e_2$ are edges) is by using a connection to $X$-joins as defined
in Section~\ref{sec:claw-free}. For each edge $e$ of $G$, fix an
endvertex $u_e$ of degree at least 4 in $G$ (which exists since $G$ is
essentially 5-edge-connected). If $e_1$ and $e_2$ are edges, set
\begin{equation*}
  X(e_1,e_2) =
  \begin{cases}
    \Setx{u_{e_1},u_{e_2}} & \text{if $u_{e_1} \neq u_{e_2}$},\\
    \emptyset & \text{otherwise.}
  \end{cases}
\end{equation*}

Suppose now that the graph $G - e_1 - e_2$ happens to contain a
connected $X(e_1,e_2)$-join $J$ spanning all of $V_{\geq 4}$. By the
classical observation of Euler, all the edges of $J$ can be arranged
in a trail $T_J$ whose first edge is incident with $u_{e_1}$ and whose
last edge is incident with $u_{e_2}$. Adding $e_1$ and $e_2$, we
obtain an $(e_1,e_2)$-trail $T$ in $G$. (If $u_1 = u_2$, we use the
fact that $u_1$ is incident with an edge of $T_J$.) Since $G$ contains
no adjacent vertices of degree 3, $T$ is an internally dominating
$(e_1,e_2)$-trail.

Summing up, the Hamilton-connectedness of $L(G)$ will be established
if we can show that for every $e_1,e_2\in E(G)$, the graph $G-e_1-e_2$
contains a connected $X(e_1,e_2)$-join spanning $V_{\geq 4}$.

How to find such $X(e_1,e_2)$-joins? Recall that in
Section~\ref{sec:claw-free}, the existence of a connected dominating
eulerian subgraph of $G$ (a connected dominating $\emptyset$-join) was
guaranteed by Lemma~\ref{l:eulerian} based on the assumption that $H$
contains an even quasitree. As shown by Lemma~\ref{l:ext}, an even
quasitree in $H$ exists whenever $H$ contains a quasitree with tight
complement. A rather straightforward modification of the proofs of
these two lemmas (which we omit) leads to the following
generalization:

\begin{lemma}\label{l:ext-join}
  Let $H'$ be a 3-hypergraph containing a quasitree $\pi$ with tight
  complement, and let $X\subseteq V(H')$. Then there is a quasigraph
  $\tau$ such that $E(\pi)$ and $E(\tau)$ are disjoint, and $\pi^* +
  \tau^*$ is a connected $X$-join in $\gr{H'}$ spanning all vertices in
  $V(H')$.
\end{lemma}
Roughly speaking, Lemma~\ref{l:ext-join} will reduce our task to
showing that for each pair of edges $e_1,e_2$ of $G$, a suitably
defined 3-hypergraph $H'$ admits a quasitree with tight complement.

Let us define the 3-hypergraph $H'$ to which Lemma~\ref{l:ext-join} is
to be applied. Suppose that $e_1$ and $e_2$ are given edges of $G$,
and let $w_i$ ($i=1,2$) be the endvertex of $e_i$ distinct from $u_i$.
We distinguish two cases:
\begin{enumerate}[\quad(1)]
\item if $e_1$ and $e_2$ have a common vertex of degree 3 (namely, the
  vertex $w_1 = w_2$), then $H'$ is obtained from $H$ by removing the
  3-hyperedge corresponding to $w_1$;
\item otherwise, $H'$ is the hypergraph obtained by performing the
  following for $i=1,2$:
  \begin{enumerate}
  \item[(2a)] if $w_i$ has degree 3, then the 3-hyperedge $e_{w_i}$ of $H$
    correponding to $w_i$ is replaced by the 2-hyperedge $e_{w_i} -
    \Setx{u_i}$,
  \item[(2b)] otherwise, the 2-hyperedge $e_i$ of $H$ is deleted.
  \end{enumerate}
\end{enumerate}

By Lemma~\ref{l:ext-join} and the preceding remarks, it suffices to
show that $H'$ admits a quasitree with tight complement. To do so, we
apply to $H'$ the proof of Theorem~\ref{t:spanning-tree}, which works
well as far as equation~\eqref{eq:both}. However, the
inequality~\eqref{eq:verts} may fail since $H'$ is not necessarily
4-edge-connected. It has to be replaced as follows. 

For an arbitrary hypergraph $H^*$, let $s(H^*)$ be the sum of all
vertex degrees in $H^*$. Let $\PP$ be the partition of $V(H')$
obtained in the proof of Theorem~\ref{t:spanning-tree}. Furthermore,
let $n^*_4$ be the number of vertices of degree 4 in $H/\PP$, and let
$n^*_{5^+} = n - n^*_4$. (All the symbols such as $n$, $m$, $m_3$ etc.,
used in the proof of Theorem~\ref{t:spanning-tree}, are now related to
the hypergraph $H'$ rather than $H$.)

It is not hard to relate $s(H')$ to $s(H)$. Indeed, the operations in
cases (1), (2a) and (2b) above decrease the degree sum by 3, 1 and 2,
respectively. It follows that $s(H') \geq s(H) - 4$ and, in fact,
\begin{equation*}
  s(H'/\PP) \geq s(H/\PP) - 4.
\end{equation*}
Since $H$ is 4-edge-connected, we know that
\begin{equation*}
  s(H/\PP) \geq 4n^*_4 + 5n^*_{5^+}
\end{equation*}
and thus we can replace~\eqref{eq:verts} by
\begin{equation*}
  4n^*_4 + 5n^*_{5+} - 4 \leq s(H'/\PP) = 2m + m_3 + \barr{m_3}.
\end{equation*}
This eventually leads to
\begin{equation*}
  m_3 \geq \barr{m_3} + n^*_{5^+} + 2
\end{equation*}
as a replacement for \eqref{eq:star}. Thus, the contradiction is much
the same as before, since we have (by the same argument as in the old
proof) that $m_3 \leq n^*_{5^+}$. This proves Theorem~\ref{t:ham-conn}
in the case of line graphs.

If $G$ is a claw-free graph, we will use a closure operation
again. However, the claw-free closure described in
Section~\ref{sec:claw-free} is not applicable, since the closure of
$G$ may be Hamilton-connected even if $G$ is not. Instead, we use the
\emph{$M$-closure} which was defined in~\cite{bib:RV-line} and applied
there to prove that 7-connected claw-free graphs are
Hamilton-connected. Let us list its relevant properties \cite[Theorem
9]{bib:RV-line}:
\begin{theorem}
  If $G$ is a connected claw-free graph, then there is a well-defined
  graph $cl^M(G)$ with the following properties:
  \begin{enumerate}[\quad(i)]
  \item $G$ is a spanning subgraph of $cl^M(G)$,
  \item $cl^M(G)$ is the line graph of a multigraph $H$,
  \item $cl^M(G)$ is Hamilton-connected if and only if $G$ is
    Hamilton-connected. 
  \end{enumerate}
\end{theorem}
Using this result (and the fact that parallel edges are allowed
throughout our argument), it is easy to prove Theorem~\ref{t:ham-conn}
just like Corollary~\ref{cor:claw-free} is proved using the claw-free
closure.


\section{Conclusion}
\label{sec:conclusion}

We have developed a method for finding dominating eulerian subgraphs
in graphs, based on the concept of a quasitree with tight
complement. Using this method, we have made some progress on
Conjecture~\ref{conj:thomassen}, although the conjecture itself is
still wide open. It is conceivable that a refinement in some part of
the analysis may improve the result a bit --- perhaps to all
5-connected line graphs. On the other hand, the 4-connected case would
certainly require major new ideas. For instance, the preimage $G$ of a
4-connected line graph may be cubic, in which case we do not even know
how to associate a 3-hypergraph with $G$ in the first place.

As mentioned in Section~\ref{sec:intro}, a simpler variant of our
method yields a short proof of the tree-packing theorem of Tutte and
Nash-Williams. It is well known that spanning trees in a graph $G$ are
the bases of a matroid, the cycle matroid of $G$, and thus matroid
theory provides a very natural setting for the tree-packing
theorem. Interestingly, quasitrees with tight complement do not quite
belong to the realm of matroid theory, although quasitrees themselves
do. Is there an underlying abstract structure, more general than the
matroidal one, which forms the `reason' for the existence of both
disjoint spanning trees in graphs, and quasitrees with tight
complement in hypergraphs?

It remains a question for further research whether our approach may be
useful for other problems on the packing of structures similar to
spanning trees, but also lacking their matroidal properties. These
include the packing of Steiner trees
\cite{bib:Kri-edge-disjoint-03,bib:Kri-edge-disjoint-09} or $T$-joins
\cite{bib:DS-packing,bib:Riz-indecomposable}.


\section*{Acknowledgments}

We thank Zden\v{e}k Ryj\'{a}\v{c}ek for many discussions of
hamiltonicity and the properties of claw-free graphs. A number of
other colleagues at the department influenced our ideas on these
topics as well. The first author would also like to acknowledge the
inspiring influence of some of the places where this research was
undertaken, particularly Polenztal in Germany; Doma\v{z}lice,
\'{U}ter\'{y} and Jahodov in Czech Republic; and Terchov\'{a} in
Slovakia.

We would like to thank an anonymous referee who read the paper
carefully and suggested a number of improvements and corrections. 


\end{document}